\documentclass[10pt,oneside]{article}
\usepackage{amsfonts, eucal, amsmath, amsthm, bbold, graphicx,color}
\usepackage[font=small,labelfont=bf]{caption}

\newtheorem{Theorem}{Theorem}
\newtheorem{Lemma}[Theorem]{Lemma}
\newtheorem{Definition}{Definition}
\newtheorem{Proposition}[Theorem]{Proposition}
\newtheorem{Corollary}[Theorem]{Corollary}

\makeatletter
\renewcommand{\@evenhead}{\sc\thepage
\centerline{ F. Fass\`o, L.~C. Garc{\'\i}a-Naranjo and N. Sansonetto }}
\renewcommand{\@oddhead}{\sc 
\centerline{ Moving energies in nonholonomic systems ---
revised version }
\hfill \thepage}
\makeatother

\makeatletter
\renewcommand\subsection{\@startsection {subsection}{1}{\z@}%
                                   {3.5ex \@plus 1ex \@minus .2ex}%
                                   {-1em}%
                                   {\normalfont\bfseries}}
\makeatother

\newcommand{\for}[1]{(\ref{#1})}
\newcommand\ugarr{\!\!\!&=&\!\!\!}

\newcommand{\G}{\Gamma}

\renewcommand{\P}{\Phi}
\newcommand{\s}{\sigma}
\renewcommand{\o}{\omega}
\renewcommand{\O}{\Omega}
\renewcommand{\r}{\rho}

\newcommand{\cC}{\mathcal{C}}                 
\newcommand{\cD}{\mathcal{D}}                 
\newcommand{\cE}{\mathcal{E}}                 
\newcommand{\cM}{\mathcal{M}}                 
\newcommand{\cR}{\mathcal{R}}       
\newcommand\cS{\mathcal {S}}

\newcommand{\Liealg}{\mathfrak g}
\newcommand{\so}{\frak{so}}     
\newcommand{\I}{\mathbb{I}}     
\newcommand{\Ad}{\mathrm{Ad}} 

\newcommand{\reali}{\mathbb{R}}
\newcommand{\rdue}{\reali^2}
\newcommand{\rtre}{\reali^3}
\newcommand{\renne}{\reali^n}

\newcommand{\der}[2]{\frac{\partial#1}{\partial#2}}
\newcommand{\dder}[3]{\frac{\partial^2#1}{\partial#2 \partial#3} }

\newcommand{\beq}[1]{\begin{equation}\label{#1}}
\newcommand{\eeq}{\end{equation}}

\newcommand\bList{
\begin{list}{}{\leftmargin2em\labelwidth1.5em\labelsep.5em\itemindent0em
\topsep.3ex\itemsep-.4ex} }
\newcommand\eList{\end{list}}

\newcommand\tg[1]{#1^{TQ}}
\newcommand\tgG[1]{#1^{TG}}
\newcommand\Emov[2]{E_{#1,#2}}
\newcommand\gpairing[2]{\langle #1,#2\rangle_{\Liealg^*\textrm{-}\Liealg} }
\newcommand\ip[2]{\langle #1,#2\rangle_{\Liealg} }
\newcommand\PS[2]{#1\cdot #2}
\newcommand\Left{L}
\newcommand\Right{R}
\newcommand\SO[1]{\mathrm{SO}(#1)}
\newcommand\SE[1]{\mathrm{SE}(#1)}
\newcommand\C{C}
\newcommand{\Fs}{\Sigma_s}
\newcommand{\Fb}{\Sigma_b}

\begin{document}

\title{\Large \bf Moving energies as first integrals \\
of nonholonomic systems with affine constraints}

\author{\normalsize
Francesco Fass\`o\footnote{\footnotesize Universit\`a di Padova,
Dipartimento di Matematica ``Tullio Levi Civita'',
Via Trieste 63, 35121 Padova, Italy.
Email: {\tt fasso@math.unipd.it} }
, 
Luis C. Garc\'{\i}a-Naranjo\footnote{\scriptsize
Departamento de Matem\'aticas y Mec\'anica,
IIMAS-UNAM,
Apdo Postal 20-726,  Mexico City,  01000, Mexico. Email:
\tt{luis@mym.iimas.unam.mx} } \
and
Nicola Sansonetto\footnote{\footnotesize Universit\`a di
Verona,
Dipartimento di Informatica, Strada le Grazie 15, 37134 Verona, Italy.
Email: {\tt nicola.sansonetto@univr.it} }
}

\date{\small{ \today}}
\vskip 1truecm
\maketitle

{\small
\begin{abstract} In nonholonomic mechanical systems with constraints
that are affine (linear nonhomogeneous) functions of the velocities,
the energy is typically not a first integral. It was shown in 
\cite{FS2016} that, nevertheless, there exist modifications of the
energy, called there moving energies, which under suitable conditions
are first integrals. The first goal of this paper is to study the
properties of these functions and the conditions that lead to their
conservation. In particular, we enlarge the class of moving energies
considered in \cite{FS2016}. The second goal of the paper is to
demonstrate the relevance of moving energies in nonholonomic
mechanics. We show that certain first integrals of some well known
systems (the affine Veselova and LR systems), which had been detected
on a case-by-case way, are instances of moving energies. Moreover, we
determine conserved moving energies for a class of affine systems on
Lie groups that include the LR systems, for a heavy
convex rigid  body that rolls without slipping on a uniformly
rotating plane, and for an $n$-dimensional generalization of the
Chaplygin sphere problem to a uniformly rotating
hyperplane.
\end{abstract}
}

\vskip3mm
{\scriptsize
}

\section{Introduction} 

\subsection{Moving energies. }
Conservation of energy in time-independent mechanical systems with
nonholonomic constraints is an important feature that has received
extended consideration. It is well known that the energy is conserved
if the nonholonomic constraints are linear (or more generally
homogeneous) functions of the velocities \cite{pars,NF} 
and that this typically does not happen if the constraints are
arbitrary nonlinear functions of the velocities (see e.g. \cite{BKMM,
marle2003, kobayashi-oliva}). The situation is better understood in
the case of systems with nonholonomic constraints that are  {\it
affine} (namely, linear non-homogeneous) functions of the
velocities. This case is important in mechanics because it is
encountered in systems formed by rigid bodies that roll without
sliding on surfaces that move in a preassigned way; instances of these
systems have been considered, e.g., by Routh \cite{routh} and more
recently in \cite{BorMamK}. This is the case that we consider in this
paper. 

The conditions under which the energy is conserved in nonholonomic
systems with affine constraints have been clarified in \cite{FS2015},
and are very special. However, it was noticed and proved
in \cite{FS2016} that, in such systems, when the energy is not
conserved, there may exist modifications of it which are conserved.
Such functions were called {\it moving energies} in \cite{FS2016}
because they were there constructed by means of time-dependent
changes of coordinates that transform the nonholonomic system with
affine constraints into a nonholonomic system with {linear}
constraints. If time-independent, the transformed system has a
conserved energy, and the moving energy is the pull-back of this
function to the original coordinates---namely, the energy of the
transformed system written in the original coordinates. This moving
energy is always a conserved function for the original system, but
the interesting case is when it is time-independent. In
\cite{FS2016}, the time-independence of the moving energy was linked
to the presence of symmetries of the system. In such a case the
conserved moving energy is the sum of two non-conserved functions:
one is the energy, and the other is the momentum of an infinitesimal
generator of the symmetry group.

This mechanism has an elementary mechanical interpretation for
systems that consist of rigid bodies constrained to roll without
sliding on moving surfaces: here the moving energy is the energy of
the system relative to a moving reference frame in which the surface
is at rest---so that the no-sliding constraint is linear---written
however in the original coordinates. Most of the examples of
moving energies produced so far \cite{FS2016,BMB2015} are indeed of
this type.\footnote{Reference \cite{BMB2015}, that refers to the
moving energy as Jacobi integral (see also Section \ref{ss:2.3}),
claims that this mechanism can be extended to less
symmetrical situations. It is unclear to us to which extent this
goal can be achieved: without symmetry, the moving energy exists but
is usually time-dependent. For more precise comments see footnote nr.
2 of \cite{FS2015}. We note that while in \cite{FS2016} and in the
present article the moving energy is regarded as a first integral of
the original system, and is therefore written in the original
coordinates, in \cite{BMB2015} it is written in the new,
time-dependent coordinates.}

The existence of this energy-like first integral may play an
important role. Reference \cite{FS2016} proved its existence (without
however determining it) in the system formed by a heavy homogeneous
sphere that rolls without sliding inside an upward convex surface of
revolution that rotates uniformly. The existence of this integral,
together with that of two other first integrals that had been
previously determined in \cite{BorMamK}, implies the integrability of
this system, in the sense of quasi-periodicity of the dynamics
\cite{FS2016,DallaVia2017}. The expression of the moving energy for
a homogeneous sphere that rolls on an arbitrary rotating surface was
given in \cite{BMB2015}, and there used to show integrability by the
Euler-Jacobi theorem (which is weaker than quasi-periodicity) 
in the case of axisymmetric surfaces. 
Other systems, whose conserved moving energy is given in \cite{BMB2015}, 
include the Chaplygin sleigh on a rotating plane and the special case of the 
affine Suslov problem considered in \cite{LGN2014} where the axis of forbidden 
rotations is also an axis of symmetry of the body.
The authors of \cite{BMB2015} also remark
that a known first integral of the  Veselova system
\cite{Veselova,Veselov-Veselova-1986} is an instance of a moving
energy.

It is therefore important to understand how general and effective this
mechanism can be, and the aim of the present paper is to investigate
this question. 

\subsection{Aim of the paper. } 
First, we will extend in a natural way the notion of moving energy,
going beyond the relation to a time-dependent change of coordinates.
Specifically, we define here a {moving energy} as the difference
between the energy of the system and the momentum of a vector field
defined on the configuration manifold of the system. This will allow a
clearer, simpler and more general treatment. 

We will investigate which vector fields produce a conserved moving
energy (Proposition \ref{thm1}) and how this relates to the existence
of symmetries of the Lagrangian (Corollary
\ref{Cor:Sym}).\footnote{While the writing of this article was almost
completed, we were informed of the existence of
the very recent article \cite{jovanovic}. Following the approach in
\cite{jovanovic2016}, this article considers
moving energies from the point of view of Noether symmetries for
time-dependent systems---calling them Noether integrals---and
generalizes to this context some of the results of \cite{FS2016}; in
particular, it proves a statement analogous to Corollary~\ref{Cor:Sym}.} 
Then, we will investigate 
some properties of moving energies, including their nonuniqueness 
(Propositions~\ref{p:2MovEn} and~\ref{p:2MovEn=}).

We will also compare this extended
notion of moving energy to the one considered so far and described
above. We will say that a moving energy is {\it kinematically
interepretable} if it arises from a time-dependent change of
coordinates, as in the mechanism described above, and we will
characterize the moving energies which are kinematically interpretable
(Proposition \ref{p:MovEn-KI}). 

Next, we will show that certain known first integrals of
some important nonholonomic systems with affine constraints are
instances of moving energies. Specifically, we will
consider a class of affine nonholonomic systems on Lie groups which
includes the affine Veselova system
\cite{Veselova,Veselov-Veselova-1986} and the more
general affine LR systems introduced in
\cite{Veselov-Veselova-1988}. 

Finally, we will determine the explicit form of
a conserved moving energy for two other important
nonholonomic systems, for which the existence of an energy-like first
integral was so far unknown: a convex body that rolls on a rotating
plane and the $n$-dimensional Chaplygin sphere that rolls
without slipping on a rotating hyperplane.

We will also indicate some dynamical consequences
of their existence (Corollary \ref{c:UnboundedMotions}) and remark
their usage for the Hamiltonization of reduced systems (see the
Remark at the end of section \ref{ss:LR}). 

The resulting picture is that the notion of moving energy is a
unifying concept in the study of nonholonomic systems with affine
constraints. Our view is that this class of functions---rather than
the energy itself---should be considered the primary `energy-like'
first integrals to be considered in these systems.

\subsection{Outline of the paper. } 
In section 2 we recall the general framework for mechanical
systems with affine nonholonomic constraints and review some of
their properties, that are needed in the subsequent study. In
particular,  given the role played by symmetry in the conservation of
moving energies, we give there a `Noether theorem' for nonholonomic
systems with affine constraints that extends previous formulations
(Proposition~\ref{prop5}). In section 3 we introduce the moving
energies, give conditions for their conservation, and
analyze some of their properties. In section~4 we investigate the
relation between the notion of moving energy introduced in section~3
and the original definition of \cite{FS2016}.

Sections 5, 6 and 7 are devoted to show that  the
aforementioned examples possess a moving energy and to compute it
explicitly. We point out that this analysis
only requires the definitions  presented in section 2 and
the results of subsection 3.1. In fact, the  reader who is interested
in applying the methods of this paper  to specific examples of
nonholonomic systems with affine constraints  need only concentrate
on these sections.

The Appendix is devoted to some additional material relative to the 
two systems studied in Sections 6 and 7. In both cases we identify a
symmetry group and obtain the reduced equations of motion.

Throughout the work, we assume that all objects (functions, manifolds,
distributions, etc.) are smooth and that all vector fields are
complete. If $\cE$ is a distribution over a manifold $Q$, then 
$\G(\cE)$ denotes the space of sections of $\cE$.

\section{Nonholonomic systems with affine constraints }

\subsection{The setting. }
We start with a Lagrangian system with $n$-dimensional configuration
manifold ${Q}$ and Lagrangian ${L}:T{Q}\to\reali$, that describes a
holonomic mechanical system. We assume
that the Lagrangian has the mechanical form 
\begin{equation}\label{hatL}
  {L}={T}+{b} - {V}\circ\pi \,,
\end{equation}
where ${T}$ is a Riemannian metric on $Q$, ${b}$ is a 1-form on ${Q}$
regarded as a function on $T{Q}$, $V$ is a function on $Q$ and $\pi:
T{Q} \to {Q}$ is the tangent bundle projection. We interpret $T$ as
the kinetic energy, $V$ as the potential energy of the positional
forces that act on the system, and the 1-form $b$ as the generalized
potential of the gyrostatic forces that act on the system.

We add now the nonholonomic constraint that, at each point ${q}\in
{Q}$, the velocities of the system belong to an affine subspace
$\cM_q$ of the tangent space $T_ qQ$. Specifically, we assume that
there are a nonintegrable distribution $\cD$ on $Q$ of constant rank
$r$, with $1<r<n$, and a vector field $Z$ on $Q$ such that, at each
point $q\in Q$,
\beq{ConstrDistr}
  \cM_ q=Z(q)+\cD_ q \,.
\eeq
Note that the vector field $Z$ is defined up to a section
of~$\cD$. The affine distribution $\cM$ with fibers
$\cM_ q$ may also be regarded as a submanifold
$M\subset TQ$ of dimension $n+r$. This submanifold is an
affine subbundle of $TQ$ of rank $r$ and will be called the {\it
constraint manifold}. The case of linear constraints is recovered
when the vector field $Z$ is {\it horizontal} (namely, it is
a section of the distribution $\cD$), since then $\cM=\cD$. 

We assume that the nonholonomic constraint is ideal, namely, that it
satisfies d'Alembert principle: when the system is
in a configuration $q\in Q$, then the set of reaction forces
that the nonholonomic constraint is capable of exerting  coincides
with the annihilator $\cD_ q^\circ$ of $\cD_ q$ 
(see e.g. \cite{pagani91,marle2003}). Under this hypothesis
there is a unique function $R:M\to\cD^\circ$, which is
interpreted as associating an ideal reaction force $ R(v_q)$ to each
constrained kinematic state $v_q\in M$, with the property
that the equations of motion of the system are given by the 
restriction to $M$ of Lagrange equations with reaction
forces $R$; for a detailed proof, see \cite{FS2015}. We will denote
$(L,Q,\cM)$ the nonholonomic system determined by these data.

In bundle coordinates $(q,\dot q)$ on $TQ$ the Lagrangian
$L$ has the form
\begin{equation}\label{Lagr}
  L(q,\dot q) = \frac12\dot q\cdot
  A(q)\dot q+b(q)\cdot \dot q - V(q)
\end{equation}
with $A(q)$ an $n\times n$ symmetric nonsingular matrix and
$b(q)\in\reali^n$. (In order to keep the notation to a minimum we do not
distinguish between global objects and their coordinate
representatives). Here, and in all expressions written in
coordinates, the dot denotes the standard scalar product in
$\renne$. In bundle coordinates, the fibers of the distribution $\cD$
can be described as the null spaces of a $q$-dependent $k\times n$
matrix $S(q)$ that has everywhere rank $k$, with $k=n-r$: 
$$
   \cD_q=\{\dot q\in T_qQ \,:\; S(q)\dot q=0\} \,.
$$
In turn
$
\cM_q = \{\dot q\in T_qQ \,:\; S(q)(\dot q-Z(q))=0\} 
$
and 
$$
  M = \big\{ (q,\dot q) \,:\, S(q)\dot q + s(q) = 0 
  \big\}
$$
with
$$
   s(q) = -S(q) Z(q) \in \reali^n\,.
$$
In coordinates, the equations of motion of the  nonholonomic
mechanical system $(L,Q,\cM)$ are
\begin{equation}
\label{eq:LagrEq}
  \Big( \frac{d}{dt}\frac{ \partial L }{\partial \dot q} 
  - 
  \frac{\partial L}{\partial q}\Big) \Big|_M \, = \, R|_M 
\end{equation}
with
\begin{equation}\label{eq:RF1}
  R=S^T(SA^{-1}S^T)^{-1} ( SA^{-1} \ell - \s )
\end{equation}
where $\ell\in\reali^n$ and $\s\in\reali^k$ have components
\begin{equation}\label{eq:RF2}
  \ell_i = \sum_{j=1}^n \dder L {\dot q_i}{q_j}\dot q_j 
  - \der L {q_i}
  \,,\qquad 
  \sigma_a = 
  \sum_{i,j=1}^n \der{S_{ai}}{q_j} \dot q_i\dot q_j + 
  \sum_{j=1}^n \der{s_{a}}{q_j} \dot q_j 
\end{equation}
($i=1,\ldots,n$, $a = 1,\ldots, k$). For details see
\cite{FS2015};  in the case of linear constraints, these or analogue
expressions are given in \cite{ago,benenti,FRS}. We note that the
restriction of $R$ to $M$ is independent of the arbitrariness that 
affects the choices of the vector field $Z$, of the matrix $S$ and
of the vector $s$, see~\cite{FS2015}.

\subsection{The reaction-annihilator distribution. }
We need to introduce now the so-called {reaction-annihilator 
distribution} $\cR^\circ$, from \cite{FRS,FS2015}. This object plays
a central role in the conservation of energy and of moving 
energies of nonholonomic systems with affine constraints 
\cite{FS2015,FS2016} (as well as
in the conservation of momenta in nonholonomic systems with either
linear or affine constraints \cite{FRS,FS2015}).

The observation underlying the consideration of this object is that,
while the condition of ideality assumes that, at each point $ q\in
Q$, the constraint is capable of exerting all reaction forces that
lie in $\cD_ q^\circ$, expression (\ref{eq:RF1}) shows that,
ordinarily, only a subset of these possible reaction forces is
actually exerted in the motions of the system. Specifically, in
bundle coordinates, $\cD_ q^\circ$ is the orthogonal complement to
$\ker S(q)$, namely the range of $S(q)^T$, but the map
$$
   S^T(SA^{-1}S^T)^{-1} (SA^{-1} \ell - \s ) \big|_{M_q} 
   \;:\;
   M_q\to\mathrm{range}\big[ S(q)^T\big]
$$
need not be surjective. Instead, the reaction forces that the
constraint exerts, when the system $(L,Q,\cM)$ is in a configuration
$q\in Q$ with any possible velocity in $\cM_ q$, are the elements of
the set
$$
  \cR_q:= 
  \bigcup_{v_q\in \cM_ q}
  R(v_q) \,,
$$
which is a subset of $\cD_q^\circ$---and typically a proper subset of
it. For instance, an extreme case is that of a heavy
homogeneous sphere that rolls without sliding on a steady
horizontal plane: all motions consist of the ball rolling with
constant linear and angular velocitiy, as in the system without the
nonholonomic constraint \cite{pars,NF}, and hence
$\cR_q=\{0\}$ at all points $q$.

The {\it reaction-annihilator distribution} $\cR^\circ$ of the
nonholonomic system $(L,Q,\cM)$ is the distribution on $Q$ whose fiber
$\cR^\circ_ q$ at $q\in Q$ is the annihilator of $\cR_ q$. 
In other words, a vector field $Y$ on $Q$ is a section of
$\cR^\circ$ if and only if, {\it in all constrained kinematic states}
of the  the system, the reaction force does no work on it,
namely\footnote{Here and in the sequel, except in Section
\ref{ss:7.2}, $\langle \ ,\ \rangle$ denotes the cotangent-tangent
pairing. }
$$
  \langle R(v_q) , Y( q) \rangle =0 \qquad 
  \forall v_q\in T_qM ,\ q\in Q \,.
$$
This is a system-dependent condition, which is weaker than
being a section of $\cD$ because
$$
  \cD_ q \subseteq \cR_ q^\circ
  \qquad \forall \, q \in Q 
  \,.
$$

For further details and examples on the reaction-annihilator
distribution see \cite{FRS,FGS2008,FGS2009,jotz,FS2015,jovanovic} and
for a discussion of its relation to d'Alembert principle see
\cite{FS2009}. 

\subsection{Conservation of energy. }\label{ss:2.3}
The {\it energy} of the nonholonomic system $(L,Q,\cM)$ is the
restriction $E_L|_M$ to the constraint manifold $M$ of the energy
$$
    E_L:=\langle p,\cdot\rangle - L
$$
of the Lagrangian system $(L,Q)$. Here $p:TQ\to \reali$ is the
momentum 1-form generated by the Lagrangian $L$, regarded as a
function on $TQ$.
If the Lagrangian is of the form \for{hatL}, then in coordinates
$p=\der L{\dot q} = A(q)\dot q+b(q)$, and
$$
  E_L=T+V\circ\pi \,.
$$
We note that, in Lagrangian mechanics, the function $E_L$ is variously
called energy, generalized energy, Jacobi integral,
Jacobi-Painlev\'e integral, sometimes with slightly different
meanings attached to each of these terms (see the discussion in
a remark in  section 1.1 of \cite{FS2015}). We simply call it energy.

As we have already mentioned, it is well known that the energy is
always conserved if the constraints are linear in the velocities. For
affine constraints, with constraint distribution as in
\for{ConstrDistr}, the situation is as follows:

\begin{Proposition}\label{p:ConsEn} {\rm \cite{FS2015}} The energy of
$(L,Q,\cM=Z+\cD)$ is conserved if and only if
$Z\in\G(\cR^\circ)$.
\end{Proposition}

Thus, energy conservation is not a universal property of
nonholonomic systems with affine constraints. Instead, it is a
system-dependent property. In particular, note that $\cR^\circ$
depends on the potentials $b$ and $V$ that enter the Lagrangian, see 
\for{eq:RF2}. Therefore, changing the (active) forces that act on the
system---even within the class of gyrostatic and conservative
forces---may
destroy or restore the conservation of energy. For some examples of
this phenomenon, which includes e.g. a sphere that rolls inside a
rotating cylinder, see~\cite{FS2015}. An extension of Proposition
\ref{p:ConsEn} to a time-dependent setting is given in Corollary 4.2
of \cite{jovanovic}.

\subsection{Conservation of momenta of vector fields. } 
We conclude this short panoramic of nonholonomic systems with affine
constraints with some results on the conservation of momenta of vector
fields and of lifted actions. Even though we will not strictly need these
results in the sequel, they will be useful for appreciating certain
aspects of the conservation of moving energies. Moreover, we will
introduce here some notation and terminology which will be
used throughout this work.

Given a nonholonomic system $(L,Q,\cM)$, we define the {\it momentum}
of a vector field $Y$ on $Q$ as the restriction to $M$ of the function
$$
  J_Y:=\langle p,Y\rangle : TQ\to\reali 
$$
(in coordinates, $J_{Y}(q,\dot q)=\der L{\dot q}(q,\dot q)\cdot Y(q)$).
A geometric characterization of the vector fields whose
momenta are first integrals of a nonholonomic system $(L,Q,\cM)$
with affine constraints does not exist.
However, just as in the case of systems
with linear constraints, see Proposition 2 of \cite{FRS}, we may
characterize those among them which have another property as well.

Here and everywhere in the sequel we denote by $\tg Y$ the tangent
lift of a vector field $Y$ on a manifold $Q$, namely the vector field
on $TQ$ which, in bundle coordinates, is given by $\tg Y= \sum_i
Y_i\partial_{q_i} + \sum_{ij}\dot q_j \der{Y_i}{q_j}\partial_{\dot
q_i}$.

\begin{Proposition}\label{prop5}
Any two of the following three conditions imply the third:
\bList
\item[i.] $Y\in\G(\cR^\circ)$.
\item[ii.] $\hat Y(L)|_M=0$.
\item[iii.] $J_Y|_M$ is a first integral of $(L,Q,\cM)$.
\eList
\end{Proposition}

\begin{proof}
We may work in coordinates. It is understood that all functions are
evaluated in $M$, and time derivatives are along the flow of the
equations of motion \for{eq:LagrEq}. Compute
\beq{ddtJ}
  \frac{dJ_Y}{dt}
  = 
  \sum_i\frac {dp_i}{dt}Y_i+\sum_{ij}p_i\dot q_j\der{Y_i}{q_j}
  =
  \sum_i\big(\der L{q_i}+R_i\big) Y_i
  + \sum_{ij}\der{L}{\dot q_i}\dot q_j\der{Y_i}{q_j}
  =
  \hat Y(L)+R\cdot Y \,.
\eeq
From this it follows that, at each point $q\in Q$, the vanishing in all of $\cM_q$
of any two among $\frac{dJ_Y}{dt}$, $R\cdot Y$ and $\hat Y (L)$
implies the vanishing in all of $\cM_q$ of the third. But the
vanishing in all of $\cM_q$ of $R\cdot Y$ is equivalent to the fact
that  $Y$ belongs to the fiber at $q$ of $\cR^\circ$. 
\end{proof}

Let now $\Psi$ be an action of a Lie group $G$ on $Q$. For each $g\in G$
we write as usual $\Psi_{g}(q)$ for $\Psi(g,q)$.
The infinitesimal generator relative to an element $\xi\in\mathfrak
g$, the Lie algebra of $G$, is the vector field
$$
  Y_\xi := \frac d{dt}\Psi_{\exp(t\xi)}\big|_{t=0}
$$
on $Q$. The tangent lift $\Psi^{TQ}$ of $\Psi$ is
the action of $G$ on $TQ$ given by $\Psi^{TQ}_{g}(v_q) = T_q\Psi_g \cdot v_q$
for all $v_q\in TQ$ (in coordinates, $\Psi^{TQ}_g(q,\dot q) = \big(
\Psi_g(q), \Psi'_g(q)\dot q\big)$ with $\Psi'_g = \der{\Psi_g}q$).
For any $\xi\in\mathfrak g$, the $\xi$-component of the momentum map
of $\Psi^{TQ}$ is the momentum $J_{Y_\xi}$ of $Y_\xi$. 

The following consequence of Proposition \ref{prop5} extends a result
in \cite{FS2015} and is a possible
statement of a `Nonholonomic Noether theorem' for nonholonomic
systems with affine constraints:

\begin{Corollary}\label{prop4} Assume that $L$ is invariant under
$\Psi^{TQ}$, namely $L\circ\Psi_g^{TQ}=L$ for all $g\in G$. Given
$\xi\in\mathfrak g$, $J_{Y_\xi}|_M$ is a first integral of
$(L,Q,\cM)$ if and only if $Y_\xi\in\G(\cR^\circ)$.
\end{Corollary}

\section{Moving energies } 
\label{S:ME}

\subsection{Definition and conservation. } In all of this section,
$(L,Q,\cM=Z+\cD)$ is a nonholonomic system with affine
constraints and we freely use the notation and the terminology
introduced in the previous section. 
For any vector field $Y$ on $Q$ define
\[
  \Emov LY := E_L - \langle p,Y\rangle 
\]
(in coordinates, $\Emov LY=E_L- p\cdot Y = p\cdot(\dot q-Y) - L$ with
$p=\der L{\dot q}$).

\begin{Definition} \label{Def1} 
A function $f:M\longrightarrow \mathbb{R}$ is called a 
{\rm moving energy} of $(L,Q,\cM)$ if there exists a
vector field $Y$ on $Q$, called a {\rm generator of $f$}, 
such that $f$ equals the restriction of $\Emov LY$ to the constraint
manifold,
$$
   f = \Emov LY|_M \,.
$$
\end{Definition}

We note that, because of the restriction to the constraint
manifold, the generator of a moving energy is never unique. Different
vector fields may lead to the same moving energy (see
Proposition~\ref{p:2MovEn=} below).

As we have mentioned in the Introduction, the notion of moving energy
given here is an extension of that originally given in \cite{FS2016},
which has a kinematical interpretation. A comparison between the two
is done in section \ref{s:MovEn-e-CoordChanges} below. 

Obviously, the consideration of moving energies has interest only when
the energy is not conserved, namely if $Z$ is not a section of
$\cR^\circ$. The central question, then, is which vector fields $Y$ produce
conserved moving energies for a given nonholonomic system $(L,Q,\cM)$.
The situation is very similar to that of which vector fields
produce conserved momenta, see Proposition \ref{prop5}:

\begin{Proposition}\label{thm1}
Any two of the following three conditions imply the
third:
\bList
\item[i.] $Y-Z\in\G(\cR^\circ)$. 
\item[ii.] $\tg Y(L)=0$ in $M$.
\item[iii.] $\Emov LY|_M$ is a first integral of $(L,Q,\cM=Z+\cD)$.
\eList
\end{Proposition}

\begin{proof} 
We work in coordinates. All functions are
evaluated in $M$. We have
$$
  \frac{d}{dt}E_L  = 
  \frac{d} {dt}(p\cdot\dot q - L) = 
  \Big(\dot p -\der L q \Big) \cdot \dot q 
  + \Big(p - \der L {\dot q}\Big) \cdot \ddot q   =
  R\cdot \dot q 
  =
  R\cdot Z 
$$
given that $\dot q-Z \in \cD$ and $R$ annihilates $\cD$.
Therefore, by \for{ddtJ},
$$
  \frac {d\Emov LY}{dt}
  =
  R\cdot(Y-Z) + \tg Y (L) 
$$
and the proof goes as that of Proposition \ref{prop5}.
\end{proof}

Proposition \ref{thm1} does not characterize all vector fields that
generate conserved moving energies, but only those which satisfy
either one (and hence the other) of the two conditions i. and ii. 
It has some immediate consequences:

\vfill\eject

\begin{Corollary}\label{Cor:Sym} \ 
\bList
\item[i.] If $\tg Y(L)|_M=0$ then  $\Emov L{Y}|_M$ is a first integral of
$(L,Q,\cM=Z+\cD)$ if and only if $Y-Z\in\G(\cR^\circ)$. 

\item[ii.] If $\tg Z(L)|_M=0$ then
$$
   \Emov L{Z}|_M = E_L- \langle p,Z\rangle |_M
$$
is a first integral of $(L,Q,\cM=Z+\cD)$. 

\item[iii.] Assume that $L$ is invariant under the tangent lift
$\tg\Psi$ of an action $\Psi$ on $Q$, namely $L\circ\tg\Psi_g=L$ for
all $g\in G$. Then for any $\xi\in\Liealg$, $\Emov L{Y_\xi}|_M$ is a
first integral of $(L,Q,\cM=Z+\cD)$ if and only if $Y_\xi-Z
\in\G(\cR^\circ)$. 
\eList
\end{Corollary}

Statement ii. is a particular case of statement i., but we have made
it explicit because---as special as it may appear---it is precisely
the case of all the affine LR systems and of their generalizations
considered in section 4. Statement iii. formalizes the idea that,  in
presence of a symmetry group of the Lagrangian, the natural
candidates to generate conserved moving energies are the
infinitesimal generators of the group action that are sections of
$\cR^\circ$. 

\vskip2mm\noindent
{\it Remarks.} i. Since the fibers of $\cD$ are contained in those of
$\cR^\circ$, the condition in item {i.\hskip-0.5mm} of
Proposition~\ref{thm1} is independent of the arbitrariness in the
choice of the component along $\cD$ of the vector field $Z$. 

ii. Statement iii. of Corollary \ref{Cor:Sym} generalizes Theorem 2
of \cite{FS2016} in two respects: it drops the assumption of the
invariance of the distribution $\cD$ under the group action and
requires $Z-Y_\xi$ to be a section of $\cR^\circ$, not of the smaller
distribution $\cD$. As is clarified in section~\ref{s:MovEn-e-CoordChanges},
these hypotheses were present in \cite{FS2016} because they are
related to the possibility of interpreting the moving energy as the
energy in a different system of coordinates, which is the case
considered there.

\subsection{Nonuniqueness of moving energies and their generators.}

A system may have different conserved moving energies and, on
the other hand, different vector fields may produce the same moving
energy. The following Proposition is a direct consequence of the
definitions:

\begin{Proposition}\label{p:2MovEn} Consider a nonholonomic system
with affine constraints $(L,Q,\cM)$ which has a conserved moving energy
$\Emov L{Y_1}|_M$. Then, for any vector field $Y_2$ on $Q$, 
the moving energy $\Emov L{Y_2}|_M$ is conserved if and
only if $J_{Y_1-Y_2}|_M$ is a first integral of 
$(L,Q,\cM)$. 
\end{Proposition}

We analyze the second question only in the special case of a
Lagrangian $L=T-V\circ\pi$ without terms that are linear in the velocities;
the general case can be easily worked out. We denote here by $\perp$ the
orthogonality with respect to the Riemannian metric defined by the
kinetic energy $T$ and by $\langle Z\rangle$ the distribution on $Q$
generated by $Z$. 

\begin{Proposition}\label{p:2MovEn=} Assume that the Lagrangian does not
contain gyrostatic terms. Let $Y_1$ and $Y_2$ be two vector
fields on $Q$. Then $\Emov L{Y_1}|_M = \Emov L{Y_2}|_M$
if and only if $Y_1-Y_2$ is a section of
$(\cD\oplus\langle Z\rangle)^{\perp}$.
\end{Proposition}

\begin{proof}
The equality $\Emov L{Y_1}|_M = \Emov L{Y_2}|_M$ is equivalent to the
condition $\langle p,Y_1-Y_2\rangle |_M=0$. Since $p=A\dot q$, this
is in turn equivalent to the condition that, at each point $q\in Q$,
$Y_1(q)-Y_2(q)$ is $T$-orthogonal to the fiber $\cM_q$, namely, to all
tangent vectors $Z(q)+u $ with $u\in \cD_q$. Since $0\in \cD_q$, this
is equivalent to the two conditions $Y_1(q)-Y_2(q)\perp Z(q)$ and
$Y_1(q)-Y_2(q)\perp\cD_q$. It follows that $\Emov L{Y_1}|_M = \Emov
L{Y_2}|_M$ if and only if, for all $q\in Q$, $(Y_1-Y_2)(q) \perp\cD_q
\oplus \langle Z\rangle_q$.
\end{proof}

\section{Kinematically interpretable moving energies. }
\label{s:MovEn-e-CoordChanges} 

We now compare the definition of moving energy given
above with the one originally given in \cite{FS2016}.  This  is not
strictly needed to apply the results of section~\ref{S:ME} to
specific nonholonomic systems, and the reader may safely decide to
skip to the treatment of examples in sections~\ref{S:MELR},
\ref{s:ConvexBody} and \ref{s:6} if he or she desires. The discussion
is however important to understand how Definition~\ref{Def1} for a
moving energy enlarges the class of these energy-type integrals
compared to the cases considered so far \cite{FS2016,BMB2015}.

\subsection{Kinematically interpretable and horizontal moving
energies. } \label{ss:4.1}
As already explained in the Introduction, the construction of moving
energies in \cite{FS2016} is as follows. One looks for a
time-dependent change of coordinates that transforms the given
nonholonomic system with affine constraints into a nonholonomic
system with {\it linear} constraints. If time-independent, the
transformed system has a conserved energy. The moving energy of ref.
\cite{FS2016} is the energy of the transformed system written in the
original coordinates.

Let us be more precise. Let $\cC$ be a time-dependent diffeomorphism
from a manifold $U$ onto $Q$, namely a smooth map $\cC:\reali\times
U\to Q$ such that, for each $t\in\reali$, the map $\cC_t:=\cC(t,\ ) : 
U\to Q$ is a diffeomorphism. As proven in \cite{FS2016} (Proposition
1; see also Proposition 4 of \cite{FS2015}), any  nonholonomic system
with affine constraints  $(L,Q,\cM=Z+\cD)$  on the configuration
space $Q$ pull-backs, under the tangent bundle lift of $\cC$, to a
nonholonomic system with affine constraints $(\tilde L,U,\tilde\cM)$
on the configuration manifold $U$. Specifically, in coordinates, the 
Lagrangian of the transformed system is
\beq{Ltilde}
  \tilde L (u,\dot u,t) 
  =
  L \big(\cC_t(u), \cC'_t(u)\dot u +\dot\cC_t(u)\big) 
\eeq
where $\cC_t'$ is the Jacobian matrix of $\cC_t$ and
$\dot\cC_t=\der{\cC_t} t$, and the transformed, time-dependent,
constraint distribution $\tilde\cM$ has fibers
\begin{eqnarray}
 \tilde\cM_{t,u}
 \nonumber\ugarr
 \cC_t'(u)^{-1} \big[ \cM_{t,\cC_t(u)} -\dot\cC_t(u) \big] 
 \\\ugarr
 \label{Mtilde}
 \cC_t'(u)^{-1}
 \big[\cD_{\cC_t(u)} +Z(\cC_t(u)) -\dot\cC_t(u) \big] \,.
\end{eqnarray}
Let now $E^*_{L,\cC}:\reali\times TQ\to\reali$ be the push-forward of
the energy $E_{\tilde L}:\reali\times TU\to\reali$ of the Lagrangian
$\tilde L$ under the tangent bundle lift of $\cC$. The restriction of
$E^*_{L,\cC}$ to $\reali\times M$ is the ``moving energy of $(L,Q,M)$
induced by $\cC$'' of reference \cite{FS2016}; in the sequel,
however, `moving energy' will always have the meaning of
Definition~\ref{Def1}. A simple computation gives
\beq{Estar}
  E^*_{L,\cC}
  =
  E_L
  - \langle p , \dot \cC_t\circ \cC_t^{-1} \rangle
  \,. 
\eeq
Without further hypotheses, this function may be time-dependent. 

\begin{Definition} \label{Def2} A moving energy  of a nonholonomic
system with affine constraints $(L,Q,\cM)$ is {\rm kinematically
interpretable} if it has a generator $Y$  and there exists a
time-dependent  diffeomorphism $\cC:\reali\times U\to Q$ which are
such that
\begin{equation}\label{Emov=Estar}
   \Emov LY=E^*_{L,\cC}
\end{equation} 
and, with the notation just introduced, 
\bList
\item[D1.] $\tilde\cM$ is a linear distribution.
\item[D2.] $\tilde\cM$ is time-independent. 
\item[D3.] $\tilde L$ is time-independent.
\eList
(Note that \for{Emov=Estar} implies that $E^*_{L,\cC}$ is
time-independent).
\end{Definition}

In formulating this definition we have taken into account the fact
that, because of the restriction to the constraint manifold $M$, the
generator of a moving energy is not unique (see Proposition 
\ref{p:2MovEn=} above). We also note that this
definition might be weakened by requiring only, instead of
(\ref{Emov=Estar}), that  the moving energy equals $E^*_{L,\cC}|_M$,
namely $\Emov LY|_M = E^*_{L,\cC}|_M$ for any generator $Y$. The
stronger requirement (\ref{Emov=Estar}) allow us to characterize
the moving energies which  are kinematically interpretable,
instead of giving only a sufficient condition for their
kinematical interpretability. Recall that $\tg Y$ denotes
the tangent lift of a vector field $Y$ on $Q$. 

\begin{Proposition}\label{p:MovEn-KI} A moving energy of a
nonhonolonomic system with affine constraints $(L,Q,\cM=Z+\cD)$ is
kinematically interpretable if and only if  it has a generator $Y$
that satisfies the following three conditions: 
\bList
\item[P1.] $Y-Z\in\G(\cD)$.
\item[P2.] $\cD$ is $Y$-invariant (namely, $\cD_{\P^Y_t(q)} =
(\P^Y_t)'(q)\cD_q$ for all $t,q$, where $\P^Y:\reali\times Q\to Q$ is the flow
of~$Y$). 
\item[P3.] $\tg {Y}(L)=0$.
\eList
\end{Proposition}

The proof of this Proposition is given in the next subsection. 

In view of Proposition \ref{thm1}, items P1 and P3 of Proposition 
\ref{p:MovEn-KI} imply that any kinematically interpretable moving
energy is a conserved quantity. However, the comparison with
Proposition~\ref{thm1} shows that the class of kinematically
interpretable moving energies is (in general) a subclass of that of
all conserved moving energies. One reason is that the kinematic
interpretability requires the vector field $Y-Z$ to be a section of
the distribution $\cD$ and not of the (generally) larger
distribution~$\cR^\circ$. 

\begin{Definition} \label{Def3} A moving energy of $(L,Q,\cM=\cD+Z)$
is said to be {\rm horizontal} if it has a generator $Y$ such that
$Y-Z\in\G(\cD)$.
\end{Definition}

Hence, a necessary condition for a moving energy to be
kinematically interpretable is that it is horizontal. All examples of
moving energies given in \cite{FS2016,BMB2015} and in the following
sections of this work are horizontal, and all, except for possibly 
a subclass  of the ones treated in section~\ref{S:MELR-1}, are kinematically
interpretable.
It would be interesting to find systems that do not have conserved
horizontal moving energies but do have conserved moving energies.

Horizontal moving energies have other reasons of interest.
Changing the Lagrangian changes $\cR^\circ$ and---as it happens for the
energy---may destroy or restore the conservation of a moving energy.
Since the distribution $\cD$ is
independent of the Lagrangian, horizontal moving energies 
are in this respect special. 
Specifically, Proposition \ref{thm1} has the following consequence:

\begin{Corollary}\label{Cor:Noether} If $Y-Z\in\Gamma(\cD)$,
then $Y$ generates a conserved moving energy for
any nonholonomic system $(L,Q,\cM=Z+\cD)$ whose Lagrangian
satisfies $\tg Y(L)|_M=0$.
\end{Corollary}

In particular, if we have an action $\Psi$ of a Lie group $G$ on $Q$
and, for some $\xi\in\mathfrak g$, the infinitesimal generator
$Y_\xi$ is such that $Y_\xi-Z\in\Gamma(\cD)$, then $Y_\xi$
generates a conserved moving energy for all nonholonomic systems
$(L,Q,\cM=Z+\cD)$ with $\tg\Psi$-invariant Lagrangian $L$. This may
be viewed as a `Noetherian' property (in the sense of
\cite{Ortega-Ratiu,FGS2009,FGS2012}) of horizontal moving energies.

An instance of this property is encountered in the
affine Suslov problem. If the axis of forbidden rotations is also an
axis of symmetry of the body,  the system is invariant under
rotations of the body frame about this axis.  Associated to this
$S^1$-symmetry there is a preserved moving energy  that persists in
the presence of invariant potentials (see \cite{BMB2015}).

\subsection{Proof of Proposition \ref{p:MovEn-KI}.} \label{ss:4.2}
The proof rests on a few results that are refinements of results from
\cite{FS2016} (and, in one case, correct a minor error of \cite{FS2016}).

\begin{Lemma}\label{l:Estar-tindep} 
Let $\cC:\reali\times U\to Q$ be a time-dependent diffeomorphism. Then
$E^*_{L,\cC}$ is time-independent if and only if $\cC$ is, up to a
diffeomorphism $\cC_0:U\to Q$, the flow $\P^Y$ of a vector field $Y$
on $Q$,\footnote{In \cite{FS2016}, Proposition 3, it is
erroneously stated that $\cC$ is a flow, without contemplating the
possibility of the presence of a diffeomorphism $\cC_0$. We correct
this here. This error does not impair the other results of
\cite{FS2016}.} namely
$$
  \cC_t = \P^Y_t\circ \cC_0 \qquad \forall t\in\reali \,.
$$
\end{Lemma}

\begin{proof} We may work in coordinates. 
Let $Y(q,t):=\dot\cC_t\circ\cC^{-1}_t \in T_qQ$. From
\for{Estar} and from the time-independency of $E_L$ and of $p=\der
L{\dot q}$ it follows that $E^*_{L,\cC}$ is time-independent if and
only if  $\langle p(q,\dot q),\der Y t(q,t)\rangle=0$ for all $q,\dot
q,t$. Since, by the assumptions made on the Lagrangian \for{Lagr},
$\dot q\mapsto p(q,\dot q)$ is, for each $q$, a linear invertible
map $T_qQ\to T^*_qQ$, this is equivalent to $\der Yt=0$. Hence
$E^*_{L,\cC}$ is time-independent if and only if there is a vector
field $Y$ on $Q$ such that $\dot\cC_t=Y\circ\cC_t$ for all~$t$.
Define a map $\P:\reali\times Q\to Q$ through
$\P_t:=\cC_t\circ\cC^{-1}_0$ $\forall t$. Since $\P_0=\mathrm{id}$
and  $\der {\P_t}t = \dot\cC_t\circ\cC^{-1}_0 = Y\circ\P_t$ $\forall
t$, $\P$ is the flow of $Y$. \end{proof}

If $\cC_t=\P^Y_t\circ\cC_0$, then \for{Estar} gives
\beq{Estar2}
    E^*_{L,\cC} = E_L - \langle p,Y \rangle \,.
\eeq

\begin{Lemma}\label{l:Mtilde_etc} Let $\cC:\reali\times U\to Q$ be a
time-dependent diffeomorphism. Assume $\cC_t=\P^Y_t\circ\cC_0$ for
all~$t$. Then:
\bList
\item[L1.] $\tilde\cM$ is a linear distribution if and only if
$Y-Z\in\G(\cD)$.
\item[L2.] $\tilde\cM$ is time-independent if and only if $\cM$ is
$Y$-invariant (namely, $(\P^Y_t)'(q)\cM_q=\cM_{\P^Y_t(q)}$ $\forall t,q$).
\item[L3.] $\tilde L$ is time-independent if and only if $Y^{TQ}(L)=0$.
\eList
\end{Lemma}

\begin{proof} (L1) Expression \for{Mtilde} shows that $\tilde
\cM_{t,u}$ is a linear subspace of $T_uU$ if and only if, at each $t$
and $u$, the vector $Z(\cC_t(u))-\dot\cC_t(u) =  (Z-Y)(\cC_t(u))$
belongs to $\cD_{\cC_t(u)}$. 

(L2) Write $\P$ for $\P^Y$. First note that if $\cC_t=\P_t\circ\cC_0$
then $\cC_t'=(\P'_t\circ\cC_0)\cC_0'$ and
$\dot\cC_t=Y\circ\P_t\circ\cC_0=(\P'_t Y)\circ\cC_0$, where the last
equality uses the invariance of a vector field under its own flow.
Therefore, \for{Mtilde} gives
$$
  \tilde\cM_{t,u} =
  \cC'_0(u)^{-1} \P'_t(\cC_0(u))^{-1}
  \big[ \cD_{\P_t(\cC_0(u))} + Z(\P_t(\cC_0(u)))
         - \P_t'(\cC_0(u) Y(\cC_0(u))  \big]
  \qquad \forall u,t\,.
$$
Thus, using the fact that $\cC_0$ is a diffeomorphism and
$\P_0=\mathrm{id}$, the condition  of time-independence
$\tilde\cM_{t,u}=\tilde \cM_{0,u}$ $\forall t,u$ is
$$
  \P'_t(q)^{-1}\big[ \cD_{\P_t(q)} + Z(\P_t(q)) \big]
  =  \cD_{q} + Z(q)  \,,
$$
namely, $\P'_t(q)^{-1}\cM_{\P_t(q)} =\cM_q$, for all $t,q$.

(L3) Proceeding as in the proof of (L2) we see that \for{Ltilde} gives
$$
 \tilde L(u,\dot u,t)
 =
 L\Big(\P_t(\cC_0(u)),
       \P'_t(\cC_0(u))\big[ \cC'_0(u)\dot u + Y(\cC_0(u))\big]\Big)
 \quad \forall u,\dot u, t \,.
$$
It follows that $\tilde L$ is time-independent if and only if
$$
 L\Big(\P_t(\cC_0(u)),
       \P'_t(\cC_0(u))\big[ \cC'_0(u)\dot u + Y(\cC_0(u))\big]\Big)
 =
 L\big(\cC_0(u), \cC'_0(u)\dot u + Y(\cC_0(u)) \big)
 \quad \forall u,\dot u, t \,.
$$
Given that $\cC_0:U\to Q$ is a diffeomorphism and so is, for each $u$,
the map $T_uU\ni \dot u\mapsto \cC'_0(u)\dot u + Y(\cC_0(u)) \in
T_qQ$, the latter condition is equivalent to
$$
   L(\P_t(q),\P_t'(q)\dot q) = L(q,\dot q) \qquad \forall q,\dot q , t 
$$
namely, to the $\tg Y$-invariance of $L$. \end{proof}

We may now prove Proposition \ref{p:MovEn-KI}. Consider first a moving
energy $\Emov LY|_M$ of a nonholonomic system $(L,Q,\cM=Z+\cD)$ whose
generator $Y$ satisfies the three conditions P1, P2 and P3
of Proposition~\ref{p:MovEn-KI}. Let $\cC=\P^Y$,
the flow of $Y$.
By \for{Estar2}, $\Emov LY = E^*_{L,\cC}$. Therefore, on account
of item L1 of Lemma~\ref{l:Mtilde_etc}, P1 implies that 
$\cC$ satisfies D1 and, on account
of item L3 of that Lemma, P3 implies that $\cC$ satisfies D3. 
To show that $\Emov LY$ is kinematically interpretable it remains
to show that $\cC$ satisfies D2. On account of L2, this is equivalent
to the $Y$-invariance of $\cM$. By itself P2 gives only the $Y$-invariance
of $\cD$ but, as we now show, together with P1 this implies the
$Y$-invariance of $\cM$. Let us write $\P$ for $\P^Y$.
Since $\P'_tY=Y\circ\P_t$ for all $t$, and since $\cD+Z-Y=\cD$ given that
$Z-Y$ is a section of $\cD$, we have 
\begin{eqnarray*}
  \P'_t(\cM_{q})
  \ugarr
  \P'_t\big( \cD_{q}+Z(q) \big)
  \\\ugarr
  \P'_t\big( \cD_{q}+Z(q)-Y(q) \big) + Y(\P_t(q))
  \\\ugarr
  \P'_t(\cD_{q}) + Y(\P_t(q))
  \\\ugarr
  \cD_{\P_t(q)} + Y(\P_t(q))
  \\\ugarr
  \cD_{\P_t(q)} + (Y-Z)(\P_t(q)) + Z(\P_t(q))
  \\\ugarr
  \cD_{\P_t(q)} +Z\circ\P_t(q)
  \\\ugarr
  \cM_{\P_t(q)} 
\end{eqnarray*}
for all $q$ and $t$. This proves that $\Emov LY$ is kinematically
interpretable.

Conversely, assume that $\Emov LY|_M$ is kinematically intrepretable.
Then there exist an $n$-dimensional manifold $U$ and a time-dependent
diffeomorphism $\cC:\reali\times U\to Q$ such that $\Emov
LY=E^*_{L,\cC}$ and $\cC$ satisfies properties D1, D2 and D3.
Therefore, by L1 and L3, $Y$ satisfies P1
and P3. Thus $Y-Z$ is a section of $\cD$ and an argument similar to
the one just used
shows that D2, namely the time-independence of $\tilde\cM$, implies
that not only $\cM$ (as in L2) but also $\cD$ are
$Y$-independent. Hence $Y$ satisfies P2 as well.

\section{Moving energies for LR systems}
\label{S:MELR}

\subsection{A moving energy for a class of affine nonholonomic systems
on Lie groups. } \label{S:MELR-1}
As a first example, we consider here a class of nonholonomic systems
$(L,Q,\cM)$ with affine constraints whose configuration manifold $Q$ is a Lie group
$G$. This class includes the so-called (affine) LR systems, that we
will consider in the next subsection. 

As usual, we denote by
$\Left$ and $\Right$, respectively, the actions of $G$ on itself by
left and right translations.\footnote{The symbols $L$ and $R$ are also
used for other objects, but there will be no risk of confusion from
the context.} A function $f:TG\to \reali$ is {left-invariant} if it is
invariant under the lifted action $\tgG \Left$ on $TG$, namely $f\circ
\tgG \Left_g=f$ for all $g\in G$. A vector field $Y$ on $G$ is
right-invariant if $(\Right_g)_*Y=Y$ for all $g\in G$. 

Corollary \ref{Cor:Sym} has the following immediate consequence (which
we formulate in a way that takes into account the fact that the
vector field $Z$ that determines the inhomogeneous part of the affine
constraint distribution is nonunique):

\begin{Proposition}\label{p:MovEn4AlmostLRSystems}
Consider a nonholonomic system with affine constraints $(L,G,\cM)$,
where $G$ is a Lie group. Assume that the
Lagrangian $L$ is {left-invariant} and that the affine distribution
$\cM$ can be written as $\cM=Z+\cD$ with a {right-invariant} vector
field $Z$. Then, the (horizontal) moving energy
$$
   \Emov L Z|_M = E_L-\langle p,Z\rangle \big|_M
$$
is a first integral of $(L,G,\cM)$.
\end{Proposition}

\begin{proof} Since the flow of a right-invariant vector field
consists of left-translations, under the stated hypotheses we have
$\tgG Z(L)=0$ in all of $TG$. Hence, the conclusion follows either from
item ii. of Corollary~\ref{Cor:Sym} or from
Corollary~\ref{Cor:Noether}, given that $Z-Z=0 \in\G(\cD)$. 
\end{proof}

If the Lagrangian has the form \for{hatL}, then the condition of
left-invariance implies that the kinetic energy $T$ is a
left-invariant Riemannian metric on $G$, that the gyrostatic term $b$
is a left-invariant 1-form on $G$, and that the potential energy $V$ is
a constant. We give here the expression of $\Emov LZ$ in the case in
which the Lagrangian does not contain gyrostatic terms, so that $L=T$. 

We employ the left-trivialization of $TG$, namely the identification
$\Lambda:TG\to G\times\Liealg$ given by the maps
$T_g\Left_{g^{-1}}:T_gG\to T_eG\equiv\Liealg$. We write $\O$ for
$T_g\Left_{g^{-1}}\dot g$ (the `body angular velocity'; in a matrix
group, $\O=g^{-1}\dot g$). We denote $\gpairing{\ }{\ }$ the
$\Liealg^*$-$\Liealg$ pairing. On account of the left-invariance of
the Lagrangian $L=T$, its left-trivialization $T\circ\Lambda^{-1}$,
which we keep denoting $T$, is given by
$$
  T(g,\O) = \frac12 \gpairing{\I\O}{\O} 
$$
where $\I:\Liealg\to\Liealg^*$ is the positive definite symmetric
tensor determined by the kinetic energy at the group identity $e$.
Furthermore, the left-trivialization of the Legendre transformation
is the map $\Liealg\to\Liealg^*$ given by $\O\mapsto\I\O$. Given that the
vector field $Z$ is right-invariant, there exists a Lie algebra vector
$\zeta\in \Liealg$ such that
\begin{equation*}
  Z(g) = T_e\Right_g \cdot \zeta \quad \forall g \in G \,,
\end{equation*}
and its left-trivialization is $\Ad_{g^{-1}}\zeta$. Hence, the
left-trivialization of the momentum $J_Z$ is
$\gpairing{\I\O}{\Ad_{g^{-1}}\zeta}$ and that of $\Emov TZ$ is
\begin{equation}\label{MovEn4AlmostLRSystems}
  E_{T,Z}(g,\O)
  \;=\;
  \frac12\gpairing {\I\O}{\O} - \gpairing{\I\O}{\Ad_{g^{-1}}\zeta }\,.
\end{equation}

We conclude this subsection with two remarks.  First, even though the
left-invariance of the Lagrangian requires the potential energy to be a
constant, and hence $L=T+b$, on account of the remark at
the end of section~4.1 any $Z$-invariant potential energy $V$
can be included into the Lagrangian without destroying the 
existence of a conserved moving energy, which is
now the restriction to $M$ of $T+V-J_Z$. Hence, in the absence
of gyrostatic terms, 
\beq{Emov_T-V}
  \Emov {T-V}Z =\Emov TZ + V \,.
\eeq
In the special case of the LR systems, a particular class of
$Z$-invariant potentials has been identified in \cite{LGNthesis}
(section 4.6). 

In all these systems, the moving energy $\Emov LZ|_M$ is
horizontal and the generator $Z$ preserves the Lagrangian. Therefore,
by Proposition \ref{p:MovEn-KI}, if $\cD$ is $Z$-invariant then
$\Emov LZ|_M$ is kinematically interpretable, in the sense of 
Definition \ref{Def2}. A particular class of systems with this
property are the LR systems of the next subsection.

\subsection{The case of LR systems. }
\label{ss:LR}
LR systems are the subclass of the class of systems on Lie groups
considered in the previous subsection in which the affine constraint
distribution $\cM=Z+\cD$ is right-invariant, namely, not only $Z$ but
also $\cD$ is right-invariant. The right-invariance of $\cD$ means
that $\cD_{g} = T_e\Right_g\cdot\cD_e$ for all $g\in G$, or 
equivalently, that $\cD$ is the null space of a set of
right-invariant 1-forms on $G$. 

LR systems were introduced by Veselov and Veselova in 
\cite{Veselov-Veselova-1988}, who focussed mostly on the case of
linear constraints, with $\cM=\cD$, and of purely kinetic Lagrangian,
namely $L=T$. The prototype of these systems is the renowned Veselova
system \cite{Veselova,Veselov-Veselova-1986}, which describes the
motion by inertia of a rigid body with a fixed point under the
constraint that the angular velocity remains orthogonal to a
direction fixed in space (linear case) or, more generally, that the component of
the angular velocity in a direction fixed in space is constant
(affine case).

As proven in \cite{Veselov-Veselova-1988}, LR systems have
remarkable properties: the existence of an invariant measure and the
conservation of the (restriction to the constraint manifold of the)
momentum covector. A difference between the linear and the affine
cases concerns of course the energy, which is conserved in the former
case but not in the latter case. However, it was shown in
\cite{Fedorov,fedorov-kozlov,fedorov-jovanovic,LGN2007}
that the affine Veselova system,
and an $n$-dimensional generalization of it, possess a first
integral which was there regarded as an ``analog of the
Jacobi-Painlev\'e integral'' or as a ``modified Hamiltonian''.
As remarked in \cite{BMB2015}, this function turns out to be a moving
energy. In fact, in full generality,
Proposition \ref{p:MovEn4AlmostLRSystems} implies the following

\begin{Corollary} \label{c:MovEn4AlmostLRSystems} Any affine LR
system $(L,G,\cM=Z+\cD)$ has the conserved moving energy $\Emov LZ|_M$.
\end{Corollary}

In order to compare the moving energy $\Emov L Z|_M$ of Corollary
\ref{c:MovEn4AlmostLRSystems} with the first
integral found in \cite{Fedorov,fedorov-kozlov,LGN2007}, we give the
expression of $\Emov LZ$ in the special case of a Lie group $G$ for
which there is an Ad-invariant inner product in $\Liealg$, which
includes the case of $\SO n$. Since
\cite{Fedorov,fedorov-kozlov,LGN2007} did not consider gyrostatic
terms, on account of \for{Emov_T-V} we limit ourselves to $L=T$.
In general, the right-invariant affine distribution
$\cM=Z+\cD$ can be specified, via right-translations, by a set of
independent covectors $a^1, \dots, a^k \in \Liealg^*$  that span the
annihilator $\cD_e^\circ$ and by the vector $\zeta \in \Liealg$ that
specifies the vector field $Z$. Hence, the constraint can be written
\begin{equation}\label{E:ConstLRGen}
   \gpairing {a^j}{\omega - \zeta} = 0, \qquad j=1,\dots, k
\end{equation}
where $\omega= \Ad_g\Omega$ is the right-trivialization of the
velocity vector $\dot g\in T_gG$.
By means of an $\Ad$-invariant inner product $\ip\ \ $,
the covectors $a^1,\ldots,a^k$ are interpreted as
elements of $\Liealg$ and can be chosen to be orthonormal. 
Also, $Z$ can be chosen so that $\zeta$ belongs to the span of
$a^1, \dots, a^k$. Hence $\zeta=\sum_{j=1}^k \ip {a^j} \zeta a^j$ and
the left-trivialization of the momentum $J_Z$ takes the form
$\sum_{j=1}^k \ip{a^j}\zeta  \ip{\I\Omega}{\gamma^j} $, where
$\gamma^j=\Ad_g^{-1}a^j$ are the so-called Poisson vectors. In
conclusion,
\begin{equation*}
  E_{T,V} (g,\Omega) = 
  \frac12\ip {\I \Omega }\Omega - 
   \sum_{j=1}^k \ip{a^j}\zeta  \ip {\I \Omega}{\gamma^j} \,.
\end{equation*}
When $G=\SO n$, this coincides with the first integral of the
$n$-dimensional affine Veselova system given in
\cite{Fedorov,fedorov-kozlov,LGN2007}, except that  the constants
$\ip{a^j}\zeta$ are there written, using the constraint
\for{E:ConstLRGen}, as $\ip{\gamma^j}\O$.

The right-invariance of the distribution $\cD$ makes the affine LR
systems very special among those considered in the previous
subsection. In particular, as mentioned above, the restriction to
the constraint manifold of the momentum covector is conserved
\cite{Veselov-Veselova-1988}. This fact is accounted for, without any
computation, by Proposition \ref{prop5}:

\begin{Proposition}\label{P:IntegralsLR}
{\rm\cite{Veselov-Veselova-1988}} 
Consider an affine LR system $(L,G,\cM=Z+\cD$). Denote $Y_\xi$ the
infinitesimal generator of the left-action of $G$ on itself by
left-translations associated to $\xi\in\mathfrak g$. Then, for any 
$\xi \in \cD_e\subset \mathfrak g$, $J_{Y_\xi}|_M =  \langle p
,Y_\xi\rangle |_M$ is a first integral of the system.
\end{Proposition}

\begin{proof} $Y_\xi(g)=T_e\Right_g\cdot\xi$. By left-invariance of
$L$, $\tgG Y_\xi(L)=0$. By right-invariance of $\cD$, $Y_\xi$ is a
section of $\cD$. The statement now follows from Proposition \ref{prop5}. 
\end{proof}

Note that if $\cD$ is not right-invariant then the infinitesimal
generators $Y_\xi$ might be not sections of $\cD$: this is the reason
why, notwithstanding the fact that the Lagrangian has the appropriate
invariance property, the momenta  $J_{Y_\xi}|_M$ are in general not
conserved for the systems considered in the previous subsection. 

Proposition \ref{P:IntegralsLR} implies that the affine LR systems
have a multitude of conserved moving energies: for any $\xi\in\cD_e$,
$\Emov L{Z-Y_\xi}|_M= \Emov LZ|_M-J_{Y_\xi}|_M$ is a conserved moving
energy. On account of Proposition \ref{p:MovEn-KI}, they are all
kinematically interpretable and hence associated to time-dependent
changes of coordinates.

\vskip2mm\noindent
\noindent{\it Remark.} The 3-dimensional affine Veselova system allows
a Hamiltonization after reduction, in terms of a rank-four Poisson
structure, with the above moving energy
playing the role of the Hamiltonian \cite{LGN2007}.

\section{A convex body that rolls on a steadily rotating plane}
\label{s:ConvexBody}

\subsection{The system. } We consider now the system formed by a
heavy convex rigid body constrained to roll without slipping on a
horizontal plane $\Pi$, which rotates uniformly around a vertical
axis, see Figure \ref{F:rollbody}. The case in which the plane is at
rest is classical and was studied for specific geometries of the body
already by Routh \cite{routh} and Chaplygin \cite{chaplygin} (see
\cite{BorMam,CDS} for recent treatments) while, to our knowledge, the
case in which the plane $\Pi$ is rotating has been studied only in
two particular cases---that of a homogeneous sphere 
\cite{earnshaw,pars,NF,BKMM,FS2016} and that of a disk
\cite{ferrario-passerini} (which however describes the disk in
a frame that co-rotates with the
plane, in which the nonholonomic constraint is linear). Here, we
exclude the latter case because we assume that the body has a smooth
(i.e. $C^\infty$) surface. 

\begin{figure}
\centering
\includegraphics[height=6cm]{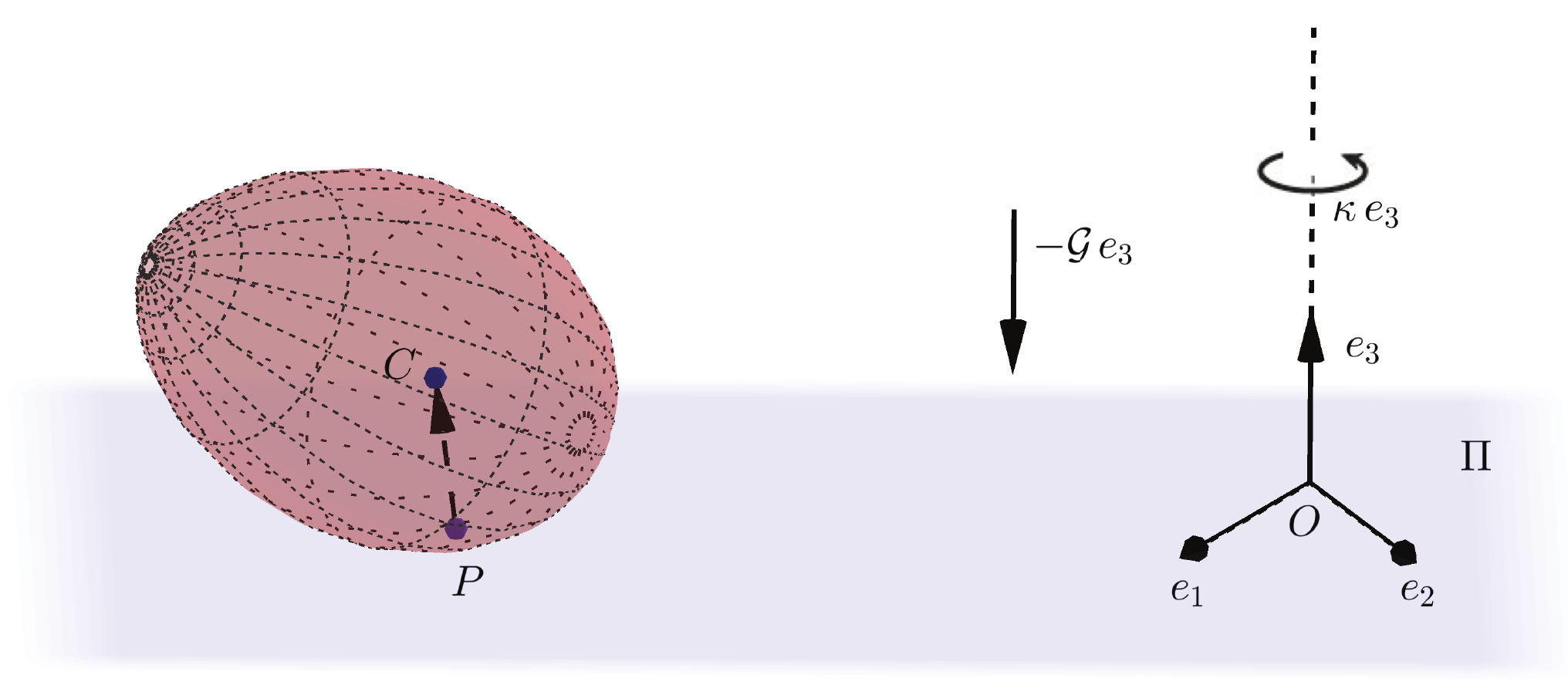}
\vskip3mm
\captionsetup{width=0.7\textwidth}
\caption{
The heavy convex rigid body with smooth surface that rolls without
slipping on a plane that rotates with constant angular
velocity~$\kappa e_3$.
}
\vskip5mm
\label{F:rollbody}
\end{figure}

We describe the system relatively to a spatial inertial frame
$\Fs=\{O;e_1,e_2,e_3\}$. We assume that the plane $\Pi$ rotates around
the axis $e_3$ of this frame, with constant angular velocity $\kappa e_3$,
$\kappa\in\reali$, and that it is superposed to the subspace spanned by
$e_1$ and $e_2$. 

Let $x=(x_1,x_2,x_3)$ be the coordinates in $\Fs$ of the
center of mass $\C$ of the body. The system is subject to the
holonomic constraint that the body has a point in contact with the
plane. The configuration manifold is thus $Q=\reali^2\times\SO 3
\ni (q,g)$, where $q=(x_1,x_2)\in\rdue$ and $g\in \SO3$ is
the attitude matrix that relates the inertial frame $\Fs$ to a frame
$\Sigma_b=\{\C;E_1,E_2,E_3\}$ attached to the body, and with the
origin in~$\C$. We assume that $g$ is chosen so that the
representatives $u^s$ and $u^b$ of a same vector in the two frames
$\Fs$ and $\Fb$ are related by $u^s=gu^b$. 

We denote by $\o$ the angular velocity of the body relative to the
inertial frame $\Fs$, and by $\O$ its representative $\o^b$ in the body
frame $\Fb$.

As in the previous section we will left-trivialize the factor $T\SO
3$ of $TQ$, but now we identify the Lie algebra $\mathrm{so}(3)$ with
$\reali^3$ via the hat map $\hat{}:\rtre\to\mathfrak{so}(3)$, with
$\hat a=a\times$ for any $a\in\reali^3$. Thus,
$TQ\equiv\rdue\times\rdue\times\SO3\times\rtre \ni (q,\dot q, g,
\O)$.

The holonomic constraint that the body touches the plane is $\PS{e_3}{OC}
=\PS{e_3}{PC}$, where $P$ is the point of the body in contact with
the plane, or $x_3=\PS{e_3}{PC}$. Hence, $\dot x_3=\PS{e_3}(\o\times
PC)$. If we denote by $\rho(g)$ the representative in $\Fb$ of the
vector $PC$ and by $\gamma(g)$ the representative 
of $e_3$ in $\Fb$, which is the so-called Poisson vector and equals 
$g^{-1}e_3^s=g^{-1}(0,0,1)^T$, then the
holonomic constraint is $x_3=\PS{\gamma(g)} {\rho(g)}$ and 
\begin{equation}\label{dotx3}
 \dot x_3 = \PS {\gamma(g)} [\O\times \rho(g)]\,.
\end{equation}

The potential energy of the weight force is $V(g)= m\mathcal{G}\PS
{e_3}{OC} = m\mathcal{G}\PS {\gamma(g)} {\rho(g)}$, with the
obvious meaning of the constants $m$ and $\mathcal G$. Here,
$\gamma$ and $\rho$ are known functions of the attitude $g\in\SO 3$,
but they are related by the Gauss map $G:\cS\to S^2$ of
the surface $\cS$ of the body. Specifically, given that $\gamma$ is the
inward unit normal vector to $\cS$ at the point of $\cS$ of coordinates
$-\rho$, $\gamma(g)=-G(-\rho(g))$. Since $\cS$ is assumed
to be smooth and convex, $G$ is a diffeomorphism and 
we may also write $\rho=F\circ \gamma$ with
$F(\gamma)=-G^{-1}(-\gamma)$. Hence
$V=v\circ\gamma$ with $v(\gamma)=m\mathcal{G}\PS {\gamma}
{F(\gamma)}$. In the sequel, we shall routinely write $\gamma$ for
$\gamma(g)$ and $\rho$ for $F(\gamma(g))$. With these conventions,
the left-trivialized Lagrangian of the system is 
\begin{equation} \label{ell}
  L(q,g,\dot q,\Omega) =
  \frac{1}{2}\PS{\I\Omega}\Omega + 
  \frac{m}{2}\Big(\dot q_1^2+\dot q_2^2 +
      \big(\PS{\gamma}[\O\times \rho]\big)^2 \Big) - v(\gamma) 
\end{equation}
and depends on the attitude $g$ only through $\gamma$. Here, $\I$ is
the inertia tensor of the body relative to its center of mass.

The condition of no-slipping of the body on the plane is obtained
by equating the velocities (relative to $\Fs$) of the point 
$P$ of the body, which is $\frac d{dt}(OC)+\o\times CP$, and of
the point of the plane which is in contact with
$P$, which is $\kappa e_3\times OP = \kappa e_3\times(OC+CP)$. Using 
representatives, the condition of no slipping is thus
\beq{dotx-anolonomo}
  \dot x = g (\Omega \times\rho) + \kappa e_3^s \times (x-g\rho)
  \,.
\eeq
The first two components of this condition define an 8-dimensional
affine subbundle $M$ of $TQ$ which can be identified with
$\rdue\times\SO3\times\rtre\ni(q,g,\O)$ (the third component of
\for{dotx-anolonomo} is nothing but \for{dotx3}).

In order to simplify the notation we identify
$Q=\rdue\times\SO3\ni(q,g)$ with its embedding in
$\rtre\times\SO3\ni(x,g)$ and 
$M=\rdue\times\SO3\times\rtre\ni(q,g,\O)$ with its embedding in
$\rtre\times\SO3\times\rtre\ni(x,g,\O)$ given, in both cases, by
$x(q,g)=(q_1,q_2,\gamma\cdot\rho)$. Correspondingly, we identify 
$TQ\equiv\rdue\times\SO3\times\rdue\times\rtre\ni(q,g,\dot q,\O)$
with its embedding in
$\rtre\times\SO3\times\rtre\times\rtre\ni(x,g,\dot x,\O)$, with $x$
as above and $\dot x=(\dot q_1,\dot q_2,\dot x_3)$ with $\dot x_3$
as in \for{dotx3}.

The affine subbundle $M$ corresponds to an affine distribution
$\cM=Z+\cD$ on $Q=\rdue\times\SO3$ which, once left-trivialized and
embedded in $\rtre\times\SO3\times\rtre\times\rtre$, is given by
\begin{equation}
\label{E:Dist-Rolling-body}
\begin{split}
   \mathcal{D}_{(q,g)}&
   =
   \left\{ \big( g(\O\times\rho) \,,\, \O\big) \,:\, \O\in\rtre \right\} 
  \\
  Z(q,g) & 
  =
  \big(\kappa e_3^s\times(x-g\rho)\,,\,  0\big) \,.
\end{split}
\end{equation}

\subsection{The conserved moving energy. } 
The energy is not conserved in the nonholonomic system $(L,Q,\cM)$
just constructed. However, being independent of $q$ and depending on
$g$ only through the Poisson vector $\gamma$, the Lagrangian
\for{ell} is invariant under the lift of an action of $\SE2$. We may
thus try to construct a moving energy using the infinitesimal
generator of the action of a subgroup.

Specifically, we consider the $S^1$-action which
(after the aforementioned embedding of $Q$ in
$\rtre\times\SO3$) is given by
\beq{S1-action}
    \theta.(x,g)=(R_\theta x, R_\theta g) 
\eeq
where $R_\theta$ is the rotation matrix in $\SO3$ that rotates an angle
$\theta\in S^1$ about the third axis $(0,0,1)^T=e_3^s$. This action
leaves the Poisson vector $\gamma$ invariant. Its lift
to $M$ (as embedded in $\rtre\times\SO3\times\rtre$) is 
\beq{lift-S1-action}
    \theta.(x,g,\O)=(R_\theta x, R_\theta g,\O) 
\eeq
and clearly leaves the Lagrangian  \for{ell} invariant. 

The infinitesimal generator of the action \for{S1-action}
that corresponds to $\xi\in
\reali\cong \mathfrak{s}^1$, once left-trivialized, is the vector field
with components
\begin{equation}\label{E:Y_kappa-rollBody}
  Y_\xi(q,g) = \big(\xi e_3^s\times  x \,,\,  \xi \gamma \big)
  \in\rtre\times\rtre \,.
\end{equation}
Therefore
$$
   Y_k-Z 
   = \big( \kappa e_3^s \times g\rho \,,\,  \kappa \gamma \big)
   = \big( g(\kappa \gamma \times \rho ) \,,\,  \kappa \gamma \big )
   \in \cD_{(q,g)} \,.
$$
Hence, the hypotheses of Corollary \ref{Cor:Sym} are satisfied and 
the moving energy $\Emov L{Y_\kappa}|_M$ is a first  integral of the system. 

In order to give an expression for this moving energy
we find convenient to introduce the vector function
\begin{equation}\label{E:DefAngMomRollBody}
  K(g,\Omega)
  =
  \I \Omega + m \rho\times \big(\Omega \times \rho\big)  \,.
\end{equation}

\begin{Proposition}\label{p:MovEnergyRollingBody} 
\begin{equation}\label{E:MovEnergyRollingBody}
   \Emov L {Y_\kappa}|_M = 
   \frac12 \PS K\O  + m\mathcal{G} \rho \cdot \gamma - 
   \kappa K\cdot \gamma+  \frac12 m\kappa^2 \big(\|\rho\|^2
   - \|x\|^2 \big) \,.
\end{equation}

\end{Proposition}

\begin{proof} 
Instead of parameterizing the embedding of $M$ in
$\rtre\times\SO3\times\rtre$ with $(x,g,\O)$, as done so far, we will
parameterize it with $(X,g,\O)$, with $X=g^{-1}x$, the representative of
$OC$ in the body frame. Specifically, $M$ is the the submanifold
of $\rtre\times\SO3\times\reali^3$ given by the condition
$X\cdot \gamma = \rho \cdot \gamma$ (the holonomic constraint).
Similarly, we parametrize $TQ\equiv
\rtre\times\SO3\times\reali^3\times\rtre$ with $(X,g,\dot x,\O)$ (notice
that we keep the spatial representative $\dot x$ of the velocity of
the center of mass). Thus, $M$ is the subbundle of $TQ$ defined by
the two conditions
\begin{equation}\label{2conditions}
  X\cdot \gamma = \rho \cdot \gamma 
  \,,\qquad 
  g^{-1}\dot x=\Omega \times \rho  + \kappa \gamma \times (X - \rho)
  \,.
\end{equation}
The energy $E_L$ of the system is the sum of the kinetic and
potential energies:
\begin{equation*}
  E_L =\frac{1}{2}\I \Omega\cdot \Omega + \frac{m}{2} \|\dot x\|^2 +
  m \mathcal{G}\gamma \cdot \rho 
\end{equation*}
Its restriction to $M$ is given by
\begin{equation*}
\begin{split}
 \left . E_L \right |_{M}
 &\;=\; 
  \frac12\I \Omega\cdot \Omega 
  + m \mathcal{G}\gamma \cdot \rho + \frac{m}{2}\|\rho \times \Omega\|^2 
  + m\kappa  \left( \Omega\times\rho \right ) 
      \cdot ( \gamma \times (X-\rho))
  + \frac{m\kappa^2}{2} \| \gamma \times (X- \rho) \|^2
  \\
  &
  \;=\;  \frac{1}{2} K\cdot \Omega  +m \mathcal{G}\gamma \cdot \rho
    - m\kappa \gamma \cdot (\rho \times( \Omega \times \rho)) 
    + m\kappa  \left ( \Omega  \times  \rho \right )
      \cdot ( \gamma \times X  )
    + \frac{m\kappa^2}{2} \| \gamma \times X \|^2
   \\ 
   & \qquad \qquad
   + \frac{m\kappa^2}{2}  \| \gamma \times  \rho \|^2  
   - m\kappa^2 ( \gamma \times X)\cdot ( \gamma \times  \rho).
\end{split}
\end{equation*}
On the other hand, given \eqref{E:Y_kappa-rollBody} and the form of
the kinetic energy metric defined by \eqref{ell}, we find
\begin{equation*}
   J_{Y_\kappa} 
   \;=\;  
   m \kappa (q_1\dot q_2 -q_2\dot q_1)+\kappa \I \Omega \cdot \gamma 
   \;=\; 
   m  \kappa  \gamma \cdot  (X \times (g^{-1}\dot x) )
     +\kappa \I \Omega\cdot  \gamma.
\end{equation*}
Its restriction to $M$ is computed to be
\begin{equation*}
  \left . J_{Y_\kappa} \right |_{M}
  \;=\; 
  m\kappa  (\Omega \times \rho) \cdot (\gamma \times X)
  - m\kappa^2(\gamma \times X)\cdot (\gamma \times \rho)
  + m\kappa^2 \| \gamma \times X\|^2 + \kappa \I \Omega \cdot \gamma
  \,.
\end{equation*}
Hence the moving energy $E_{L,Y_\kappa}|_M =\left . E_L \right |_{M}-
\left . J_{Y_\kappa} \right |_{M} $ is given by
\begin{equation*}
  E_{L,Y_\kappa}|_M
  \;=\; 
  \frac{1}{2}K\cdot \Omega + m\mathcal{G} \rho \cdot \gamma
  -\kappa (K\cdot \gamma)
  +\frac{m\kappa^2}{2}
    \left ( \|\gamma \times \rho\|^2 - \|\gamma \times X\|^2 \right )
  \,.
\end{equation*}
This is equivalent to \eqref{E:MovEnergyRollingBody} because, in $M$,
$X\cdot \gamma=\r\cdot\gamma$ and hence
$\|\gamma \times \rho\|^2 - \|\gamma \times X\|^2
= \|\rho\|^2-\|X\|^2$, and because $\|X\|=\|x\|$.
\end{proof}

We note that the existence of the moving energy $E_{L,Y_\kappa}|_M$
has the following dynamical consequence:

\begin{Corollary} \label{c:UnboundedMotions}
If the motion of the rolling body is unbounded, then its angular
velocity $\Omega$ satisfies $\limsup_{t\to \infty} \|\Omega\| =  \infty$.
\end{Corollary}
\begin{proof}
Since $\rho$ and $\gamma$ are bounded, the only way in which
the conserved moving energy \for{E:MovEnergyRollingBody} remains
bounded as $\|x\|$ becomes large is that $\|\Omega\|$ becomes large. 
\end{proof}

\vskip2mm\noindent
\noindent{\it Remark.} The distribution $\cD$ and the vector field
$Z$ are also invariant under the lifted $S^1$-action
\for{lift-S1-action}. Therefore, the system is invariant under this
action. In the Appendix, we give for
completeness the reduced equations of motion on $M/S^1$.

\section{The $n$-dimensional Chaplygin sphere that rolls on a steadily
rotating hyperplane } \label{s:6}

\subsection{The system.}
It is natural to expect that the discussion of the previous section
admits a multi-dimensional generalization. Here we 
consider the particular case in which the body is an
$n$-dimensional sphere and the center of mass coincides with
its geometric center. If the hyperplane where the rolling takes plane
is not rotating we recover the $n$-dimensional Chaplygin sphere
problem introduced in \cite{fedorov-kozlov}. 

Let $x\in \reali^n$ denote the position of the center of mass of the
sphere written with respect to an inertial frame  $\Sigma_s=\{O; e_1,
\dots ,e_n\}$. We assume that the hyperplane where the rolling takes
place passes through $O$ and has $e_n$ as its normal vector.
Moreover, we assume that  the sphere is `above' this hyperplane so at
all times  the holonomic constraint $x_n=r$ is satisfied, where $r$
is the radius of the sphere. 

The configuration space is $Q=\reali^{n-1}\times \SO n \ni (q,g)$, where
$q=(x_1,\dots, x_{n-1})$. For convenience, in all this section
we embed $\reali^{n-1} \hookrightarrow \reali^n$ by putting  $x_n=r$. We
will also work with  the induced embedding of tangent bundles  $TQ
\hookrightarrow T(\reali^n\times  \SO n )$ defined by the simultaneous
relations $x_n=r$ and $\dot x_n=0$.  

The Lagrangian $L:T(\reali^n\times  \SO n )\to \reali$ is written in the
left-trivialization as
\begin{equation}
\label{E:LagChapBall}
L(x,g,\dot x, \Omega)=\frac{1}{2}\langle \I\Omega , \Omega \rangle  +
m \|\dot x \|^2.
\end{equation}
 As usual  $\Omega=g^{-1}\dot g \in \so(n)$
is the angular velocity written in body coordinates (the
left-trivialization
of the velocity). The pairing $\langle \cdot , \cdot \rangle$ in
\eqref{E:LagChapBall} denotes the Killing metric
\begin{equation*}
  \langle \zeta_1 , \zeta_2 \rangle =
  -\frac{1}{2}\mbox{Trace}( \zeta_1  \zeta_2), 
  \qquad \zeta_1, \zeta_2\in \so(n),
\end{equation*}
and the inertia tensor $\I:\so(n)\to \so(n)$ is a positive definite
symmetric linear operator.

The steady rotation of the hyperplane where the rolling takes place 
is specified by a fixed element $\eta\in \so(n)$ that satisfies
\begin{equation}
\label{E:Cond-eta}
  \eta e_n=0.
\end{equation}
The nonholonomic constraints  of rolling without slipping are
\begin{equation}
\label{E:Constraint-n}
\dot x = r\omega e_n + \eta x,
\end{equation}
where $\omega=\dot g g^{-1}$ is the angular velocity written in space
coordinates (the right-trivialization of the velocity) that satisfies
$\omega=\Ad_g\Omega$. Note that the last component of
\eqref{E:Constraint-n} reads $\dot x_n=0$ so that
\eqref{E:Constraint-n} defines an
affine constraint subbundle of $TQ \subset T(\reali^n\times  \SO n )$.

The constraint \eqref{E:Constraint-n} may be rewritten as 
\begin{equation}
\label{E:Constraint-n-bis}
  \dot x = r(\Ad_g\Omega) e_n + \eta x,
\end{equation}
that defines an affine distribution $\cM=Z+\cD$ on $\reali^n\times\SO n$
that is given in the left-trivialization by
\begin{equation}
\label{E:Dist-Chap-ball}
\begin{split}
   \mathcal{D}_{(x,g)}&
   =
   \left\{ (\dot x,\O) \,:\, 
       \dot x =  r(\Ad_g\Omega) e_n \right\} \,,
  \\
  Z(x,g) & 
  =
  \big(\eta x, 0\big) \,.
\end{split}
\end{equation}
The constraint manifold $M$ is diffeomorphic to
$\reali^{n-1}\times\SO n\times \so(n) \subset
\reali^n\times\SO n\times \so (n)\ni (x,g,\O)$. 

\subsection{The conserved moving energy. } \label{ss:7.2}
The energy is not conserved for the above system. A conserved moving
energy may be found by considering the action of $\SO{n-1}$
on $\reali^n\times \SO n$ defined by $h\cdot (x,g)=(\tilde h x, g)$,
where for $h\in \SO{n-1}$ we denote
\begin{equation}
\label{E:SO(n)Emb}
   \tilde h = \left ( \begin{array}{cc} h & 0 \\ 0 & 1 \end{array} \right
   )\in \SO n.
\end{equation}
This action leaves $Q$ invariant and its tangent lift  clearly
preserves the Lagrangian \eqref{E:LagChapBall}.  

Using \eqref{E:Cond-eta} and the embedding $\SO{n-1}\hookrightarrow
\SO n$ given by \eqref{E:SO(n)Emb}, we can naturally think of
$\eta$ as an element in $\so(n-1)$. Its infinitesimal generator 
$Y_\eta$ is readily computed to be the vector field $Y_\eta(x
,g)=(\eta x,0)$. Therefore, $Y_\eta- Z=0$ and by Corollary
\ref{Cor:Sym} the moving energy $E_{L,Y_\eta}|_M$  is preserved.

In order to give an expression for this moving energy we
introduce the Poisson vector $\gamma(g)=g^{-1}e_n$ and
the matrix\footnote{This matrix was introduced in
\cite{fedorov-kozlov} as  the angular momentum of the sphere about
the contact point.}
\begin{equation}\label{K-2}
  K:=\I\Omega +mr^2 \big( \Gamma(g) \Omega +\Omega\Gamma(g) \big)\in \so(n),
\end{equation}
where $\Gamma(g)$ denotes the symmetric, rank one matrix
$\Gamma:=\gamma \otimes \gamma=\gamma \gamma^T$.

\begin{Proposition}\label{p:MovEnergyChapBall-nD} 
\begin{equation}\label{E:MovEnergyChapBall-nD}
   \Emov L {Y_\eta}|_M (x,g,\dot x, \Omega) = 
   \frac{1}{2} \langle K , \Omega
  \rangle -\frac{m}{2}  \|\eta x\|^2
\,.
\end{equation}

\end{Proposition}

\begin{proof} The restriction of the energy $E_L$ to $M$ is obtained
by substituting the constraint \eqref{E:Constraint-n} into the
Lagrangian \eqref{E:LagChapBall}. Notice that
\begin{equation*}
 \frac{m}{2}||  r\omega e_n + \eta x||^2=\frac{mr^2}{2}||\Omega \gamma ||^2
 +\frac{m}{2}||\eta x||^2+ mr  ( \omega e_n, \eta x ),
\end{equation*}
where $(\cdot , \cdot)$ is the euclidean scalar product in
$\mathbb{R}^n$ (so far denoted by a dot). Also,
\begin{equation*}
 \frac{1}{2}\langle \I \Omega, \Omega \rangle + \frac{mr^2}{2}||\Omega \gamma ||^2 = \frac{1}{2}\langle K , \Omega
 \rangle. 
\end{equation*}
Therefore,
\begin{equation}
 \label{E:Energy}
 \left . E_L \right |_{M}(x,g,\dot x, \Omega)
  =
 \frac{1}{2} \langle K , \Omega \rangle 
 + \frac{m}{2} ||\eta x||^2 + mr ( \omega e_n , \eta x).
\end{equation}
On the other hand, considering that $Y_\eta(x,g)=(\eta x,0)$ and the
form of the Lagrangian  \eqref{E:LagChapBall}, we have
\begin{equation*}
   J_{Y_\eta}(x,g,\dot x, \Omega)
   \;=\;    m(\dot x, \eta x).
\end{equation*}
Its restriction to $M$ is given by
\begin{equation*}
 \left .J_{Y_\eta} \right |_{M}(x,g,\dot x, \Omega)
   \;=\;    mr(\omega e_n, \eta x)+m\|\eta x\|^2.
\end{equation*}
Hence the moving energy $E_{L,Y_\eta}|_M =\left . E_L \right |_{M}-
\left . J_{Y_\eta} \right |_{M} $ is given by \eqref{E:MovEnergyChapBall-nD}.
\end{proof}

\vskip2mm\noindent
\noindent{\it Remark.} The system is invariant under the action of a
certain subgroup  of $\SO{n-1}$. The precise form of this subgroup
and its action, together with the unreduced and the reduced equations
of motion, is given, for completeness, in the Appendix.

\vfill\eject

\section{Appendix: Equations of motion for the examples}

We present here the (reduced) equations of motion of the examples
treated in sections  \ref{s:ConvexBody} and \ref{s:6}. 

\subsection{The $S^1$-reduced equations of motion of the convex
body that rolls on a rotating plane. }
\label{s:5.3}
In the system studied in section \ref{s:ConvexBody}, not only the
Lagrangian \for{ell} but also the distribution
$\cD$ and the vector field $Z$ as in \for{E:Dist-Rolling-body}
are invariant under the lift of the $S^1$-action \for{lift-S1-action}
to $TQ$. Therefore, this lifted action can be restricted to
the 8-dimensional phase space $M$, and the dynamics is equivariant.
For completeness, we give here the reduced equations of motion on the
quotient space $M/S^1$.

The $S^1$-action \for{lift-S1-action} on $M$ is free. The 7-dimensional
quotient manifold $M/S^1$ can be identified with $\rdue\times
S^2\times \rtre\ni(q,\gamma,\O)$, with projection 
$$
   (q,g,\O) \mapsto\big(q,\gamma(g),\O\big) \,.
$$
We embed $M/S^1$ in $\reali^9\ni(X,\gamma,\O)$, as the submanifold given by 
\beq{restriction}
   \|\gamma\|=1  \,,\qquad (X-\rho)\cdot\gamma=0 
\eeq 
where $\rho$ stands for $\rho=F(\gamma)$
(recall that $X= g^{-1} x$ and  see the first equation of
(\ref{2conditions})).

The definition \for{E:DefAngMomRollBody} of $K$ can be inverted to
give
\begin{equation}
\label{E:OmKGamm}
  \Omega(\gamma,K) = AK + \frac{mA\rho\cdot K}{1-m A\rho\cdot \rho} A\rho,
\end{equation}
where $A = (\I+ m \|\rho\|^2\mathbb{1} )^{-1}$ and $\rho=F(\gamma)$. 
Therefore, as (global) coordinates on $\reali^9$ we may use
$(X,\gamma,K)$ instead of $(X,\gamma,\Omega)$.

\begin{Proposition}
\label{T:RollingBody2}
The equations of motion of the $S^1$-reduced system are the restriction to
the submanifold \for{restriction} of the equations
\begin{equation}
\label{E:FullRollingBody2}
\begin{split}
  \dot K&=K \times \Omega +  m\dot \rho \times (\Omega \times \rho)
  +  m\mathcal{G} \gamma\times \rho
  + m \kappa \rho  \times  (\kappa X - \dot \rho \times \gamma  ),
\\
  \dot X &= (X-\rho )\times (\Omega-\kappa \gamma) ,
  \\
  \dot \gamma & = \gamma \times \Omega,
\end{split}
\end{equation}
on $\reali^9\ni(X,\gamma,K)$, where $\O=\O(\gamma,K)$ is as in
\for{E:OmKGamm} and $\dot \rho$ is
shorthand for $DF(\gamma)(\gamma \times \Omega)$.\footnote{It
is immediate to check that both
$\|\gamma\|^2$ and $\gamma \cdot (X - \rho)$ are first integrals of
\eqref{E:FullRollingBody2}. For the latter one should use the 
kinematic relation $\dot \rho \cdot \gamma=0$.}
\end{Proposition}

The equation for $\gamma$ is the well-known evolution equation  of
the Poisson vector $\gamma$ that can be deduced by direct
differentiation of the defining relation $\gamma=g^{-1}e_3^s$. The
evolution equation for $X$ follows by differentiating $X=g^{-1}x$ and
using the nonholonomic constraint \eqref{2conditions}. Both of these
equations are kinematical. The  equation for $K$ is a balance of
momentum. The full dynamics of the system on $M$ is obtained by
adjoining the reconstruction equation $\dot g=g\hat \Omega$.

\begin{proof}
We begin by writing the equations of motion as
\begin{equation} \label{E:MotionReactF}
   m\ddot x= -m\mathcal{G}e_3+ R_1, \qquad  \frac{d}{dt}( \I \Omega)
   = \I \Omega \times \Omega +  R_2,
\end{equation}
where $R=(R_1,R_2)$ is the nonholonomic constraint force/torque.
D'Alembert's principle states that $(R_1,R_2)$ should annihilate any
vector in the distribution $\cD$. In view of
\eqref{E:Dist-Rolling-body} one finds that $R_2=(g^{-1}R_1) \times
\rho$. On the other hand, differentiating the constraint
\eqref{2conditions} and combining it with the first of the above
equations yields
\begin{equation*}
  g^{-1}R_1
  =
  m\mathcal{G}\gamma + m\Omega\times (\Omega \times \rho)
  + m \dot \Omega \times \rho + m \Omega \times \dot \rho
  + m\kappa\gamma\times (g^{-1}\dot x-\Omega \times \rho - \dot \rho) 
\end{equation*}
Using again \for{2conditions}  and \for{restriction} this simplifies
to
\begin{equation*}
\begin{split}
   g^{-1}R_1
   &=
   m\mathcal{G}\gamma + m\Omega\times (\Omega \times \rho) +
   m \dot \Omega \times \rho + m \Omega \times \dot\rho  +
   m\kappa \dot \rho \times \gamma -m\kappa^2(X-\rho).
\end{split}
\end{equation*}
Using this expression and substituting $R_2=(g^{-1}R_1) \times \rho$ 
in the second equation of  \eqref{E:MotionReactF} gives
\begin{equation*}
\begin{split}
  \frac{d}{dt}( \I \Omega)&= \I \Omega \times \Omega
  + m\mathcal{G} \gamma \times \rho
  + m ( \Omega \times (\Omega \times \rho))\times \rho
  + m (\dot \Omega \times \rho)\times \rho
  + m(\Omega \times \dot \rho)\times \rho \\
&\qquad \qquad  
 + m \kappa \rho \times \left  ( \kappa X - \dot \rho \times \gamma \right ).
 \end{split}
\end{equation*}
A simple calculation that uses the definition of $K$ shows that the
above relation is equivalent to the first equation in
\eqref{E:FullRollingBody2}. \end{proof}

\vskip2mm\noindent
\noindent{\it Remark.} For future reference we note that if the body
is a sphere whose  center of mass coincides with its geometric center
(a Chaplygin ball) then $\rho=r\gamma$, where $r>0$ is the sphere's
radius, and  \eqref{E:FullRollingBody2} simplifies to 
\begin{equation}
\label{E:Chapball-3d}
\begin{split}
 \dot K&=K\times \Omega -mr^2\kappa \gamma \times \Omega
 + mr\kappa^2 \gamma \times X, \\
 \dot X &= (\kappa \gamma -\Omega)\times X + r\Omega\times \gamma, \\
 \dot \gamma &=\gamma \times \Omega.
\end{split}
\end{equation}
As it is checked directly, $K\cdot \gamma$ is a first integral of the
system, so conservation of the moving energy
\eqref{E:MovEnergyRollingBody} implies conservation of 
\begin{equation}\label{E:MovEnergyRollingBody3Dbis}
  \tilde E = 
   \frac12 \PS K\O  
   -\frac{m}{2}\kappa^2 \|X \|^2  \,,
\end{equation}
that is also a moving energy (see the remark at the end of the next
section).

\subsection{The equations of motion for an $n$-dimensional Chaplygin
sphere rolling on a steadily rotating hyperplane. } We continue using
the notation introduced in section~\ref{s:6}. 
Here too, \for{K-2} can be solved for $\O$ and as global coordinates
on $M$ we may use $(x,g,K)$.

\begin{Proposition}  \label{T:nChap}
The equations of motion for an $n$ dimensional Chaplygin ball that
rolls without slipping on a hyperplane that steadily rotates with
angular velocity $\eta\in \so(n)$ are given by
\begin{equation}
\label{E:MotionFinal-n}
\begin{split}
 \dot K &=  [K , \Omega] -mr (g^{-1}\eta \dot x)\wedge \gamma, \\
 \dot x &= r(\Ad_g\Omega) e_n + \eta x, \\
 \dot g & =g \Omega,
\end{split}
\end{equation}
where $[\cdot , \cdot ]$ denotes the
matrix commutator in $\so(n)$ and the wedge product of vectors $a, b
\in \reali^n$ is defined by $a\wedge b =ab^T-ba^T$.
\end{Proposition}

\begin{proof}
The second and third equation of \eqref{E:MotionFinal-n} follow,
respectively, from the constraint \eqref{E:Constraint-n-bis} and the
definition of $\Omega$. So we only need to prove that the first
equation, that is a balance of momentum, holds. Denote by
$(R_1,R_2)\in  \mathbb{R}^n \times \so(n)$ the force/torque exerted
by the constraint.  The equations of motion are
\begin{equation*}
m\ddot x = R_1, \qquad \I \dot \Omega =[ \I \Omega, \Omega] +R_2.
\end{equation*}

Differentiating \eqref{E:Constraint-n-bis} we obtain
\begin{equation}
\label{eq:auxR1}
  R_1=mr(\Ad_g \dot \Omega)e_n+m\eta \dot x.
\end{equation}
In order to determine $R_2$, we use the fact that d'Alembert's
principle implies that the reaction force $(R_1,R_2)$ annihilates 
any $(\dot x, \Omega)$ belonging to the distribution $\mathcal{D}$.
As we now show, this condition gives 
\begin{equation}
\label{eq:R2-Chap}
   R_2= -r \Ad_{g^{-1}}(R_1 \wedge e_n).
\end{equation}
Let $(\dot x, \Omega)\in \mathcal{D}_{(x,g)}$. Then $\dot
x=r(\Ad_g\Omega)e_n$, and  d'Alembert's principle implies that
\begin{equation}
\label{eq:auxR2}
    R_1\cdot (r(\Ad_g\Omega)e_n)+\langle R_2, \Omega \rangle =0.
\end{equation}
Note however that
\begin{equation*}
\begin{split}
   R_1 \cdot \left (  ( \Ad_g\Omega)e_n \right )  
   &= \frac{1}{2}\mbox{Trace} \left ( R_1( ( \Ad_g\Omega)e_n)^T \right )
   +\frac{1}{2}\mbox{Trace} \left (( ( \Ad_g\Omega)e_n) R_1^T \right ) \\
   & =-\frac{1}{2}\mbox{Trace} \left ( R_1e_n^T( \Ad_g\Omega)
   - e_nR_1^T( \Ad_g\Omega)  \right )  \\
   &=  \langle R_1\wedge e_n , \Ad_g\Omega \rangle\\
   & = \langle  \Ad_{g^{-1}}(R_1 \wedge e_n) , \Omega \rangle,
\end{split}
\end{equation*}
so \eqref{eq:auxR2} may be rewritten as
\begin{equation*}
  \langle R_2+r \Ad_{g^{-1}}(R_1 \wedge e_n) , \Omega\rangle=0.
\end{equation*}
Since $\Omega\in \so(n)$ may be chosen arbitrarily, and the pairing
$\langle \cdot , \cdot \rangle$ is non-degenerate, then
\eqref{eq:R2-Chap} holds.

Inserting \eqref{eq:auxR1} into \eqref{eq:R2-Chap} leads to
\begin{equation*}
\begin{split}
\label{E:mu-n}
  R_2=-r\Ad_{g^{-1}}(R_1 \wedge e_n)
  &= 
  - mr^2(\dot \Omega \gamma)\wedge \gamma-mr(g^{-1}\eta \dot x)\wedge\gamma
  \\
  &=
  -mr^2(\Gamma \dot \Omega +\dot \Omega \Gamma)-mr(g^{-1}\eta \dot x)
  \wedge \gamma,
\end{split}
\end{equation*}
where we have used the identity
$\Ad_{g^{-1}}(a\wedge b)=(g^{-1}a)\wedge (g^{-1}b)$, and the
definitions of $\gamma, \Gamma$. Therefore,
\begin{equation*}
\label{E:Motion2-n}
  \I \dot \Omega
  =
  [\I \Omega , \Omega ] -mr^2 (\Gamma \dot \Omega+\Gamma \dot \Omega)
  -mr (g^{-1}\eta \dot x)\wedge \gamma,
\end{equation*}
that is seen to be equivalent to the last equation in
\eqref{E:MotionFinal-n} by using the definition of $K$.
\end{proof}

A direct calculation shows that both the Lagrangian
\eqref{E:LagChapBall} and the affine  constraint distribution
determined by \eqref{E:Constraint-n} are invariant under the tangent
lift of the $H$-action on $Q$ defined by
\begin{equation*}
   h\cdot (x,g)=(hx,hg) \,,
\end{equation*}
where $H$ is the following closed, Lie subgroup of $\SO n$
\begin{equation*}
   H:=\{ h\in \SO n \, : \,  h^{-1}e_n=e_n, \quad \Ad_{h}\eta=\eta \}.
\end{equation*}
The reduced equations are conveniently written in terms of $K$,
$\gamma$, $X:=g^{-1}x$, and $\Xi:=\Ad_{g^{-1}}\eta$. These variables
are not independent but satisfy 
\begin{equation}
\label{E:InvManifold-Symm-Chap-Sphere}
  \Xi \in \mathcal{O}_\eta, \qquad \gamma \cdot X
  =
  r, \qquad \Xi \gamma=0, \qquad \|\gamma\|=1.
\end{equation}
Here $ \mathcal{O}_\eta$ denotes the adjoint orbit through $\eta \in \so(n)$, and the other relations follow from  our previous assumptions $x_n=r$, $\eta e_n=0$, and the definition of the Poisson vector $\gamma$. 

\begin{Proposition}  \label{T:nChap-Symm}
The $H$-reduction of the system \eqref{E:MotionFinal-n}  is given by
the restriction to the invariant manifold
\eqref{E:InvManifold-Symm-Chap-Sphere}  of the following system for
$(K,X,\gamma,\Xi)\in \so(n)\times \mathbb{R}^n \times \mathbb{R}^n
\times \so(n)$ 
\begin{equation}
\label{E:MotionFinal-n-Symm}
\begin{split}
 \dot K &=  [K , \Omega] -mr^2[\Xi , (\Omega \gamma) \wedge \gamma ]
 - mr [ \Xi , [ \Xi , X \wedge \gamma] ] ,\\ 
 \dot X &=  ( \Xi - \Omega) X +r \Omega \gamma, \\
 \dot \gamma &= -\Omega \gamma, \\
 \dot \Xi & = [\Xi , \Omega ].
\end{split}
\end{equation}
Moreover, the conserved moving energy \eqref{E:MovEnergyChapBall-nD}
is $H$-invariant and may be written as
\begin{equation}
\label{E:red-moving-energy-nD-ChapBall}
   \Emov L {Y_\eta}|_M (\Omega,X,\gamma,\Xi) = 
   \frac{1}{2} \langle K , \Omega
  \rangle -\frac{m}{2}  \|\Xi X\|^2
\,.
\end{equation}
\end{Proposition}
\begin{proof}
The equations for $\gamma$ and $\Xi$ are purely kinematical and follow from  the definition of $\Omega=g^{-1}\dot g$. 
The equation for $X$ follows by differentiating $X=g^{-1}x$ and using the constraint  \eqref{E:Constraint-n-bis}.
On the other hand, using again \eqref{E:Constraint-n-bis}, we find
\begin{equation*}
g^{-1}\eta\dot x= r\Xi\Omega \gamma + \Xi^2 X.
\end{equation*}
Now, given that $\Xi \gamma=0$, we have $(\Xi a)\wedge \gamma = [ \Xi, a\wedge \gamma]$ for all $a\in \mathbb{R}^n$.
Therefore 
\begin{equation*}
(\Xi\Omega \gamma)\wedge \gamma = [\Xi, (\Omega \gamma)\wedge \gamma], \qquad  (\Xi^2 X)\wedge \gamma = [\Xi, [\Xi, X\wedge \gamma]].
\end{equation*}
Whence,
\begin{equation*}
(g^{-1}\eta\dot x)\wedge \gamma = r [\Xi, (\Omega \gamma)\wedge \gamma] + [\Xi, [\Xi, X\wedge \gamma]],
\end{equation*}
and the first equation in \eqref{E:MotionFinal-n-Symm} follows from the first equation in \eqref{E:MotionFinal-n}.

Finally, the statement for the conserved moving energy follows by
noticing that $\|\eta x\|^2=\| \Xi X\|^2$.~\end{proof}

\vskip2mm\noindent
\noindent{\it Remark.} The symmetry group $H$ is a subgroup of the
copy of $\SO {n-1}$ inside $\SO n$ that fixes~$e_n$. It is in fact a
proper subgroup except when $n=3$ where $H=\SO 2$. 
An explanation for this is that in 3 dimensions  the vector that is normal to the plane where the rolling takes place is also
the axis of rotation of the plane.
 In dimension $n\geq 4$  this interpretation of the vector that is normal to the hyperplane where the rolling takes place is no longer possible
 and, consequently, the symmetry group of the problem is smaller.

In the special case $n=3$, one may identify $\Xi\in \so(3)$ with $\kappa
\gamma \in \mathbb{R}^3$, where $\kappa \in \mathbb{R}$ is the
angular speed of rotation of the plane where the rolling takes place.
One may then check that \eqref{E:MotionFinal-n-Symm} is equivalent to
\eqref{E:Chapball-3d} via the hat map isomorphism between $\so(3)$
and $\mathbb{R}^3$,  and that the conserved moving energy
\eqref{E:red-moving-energy-nD-ChapBall} differs from
\eqref{E:MovEnergyRollingBody3Dbis} by a constant.

\vskip1cm\noindent
{\bf Acknowledgments.} LGN acknowledges support of a Newton Advanced
Fellowship (ref: NA140017) from the Royal Society that, among other
things, funded a joint visit of FF and himself to the University of
Manchester in October 2016 where part of this work was completed. He
is also thankful to the University of Padova for its hospitality in
July 2016, and he acknowledges the support for his research from the
Program UNAM-DGAPA-PAPIIT-IA103815.


{\small
\begin{thebibliography}{99}

\bibitem{ago}
\newblock{C. Agostinelli,} 
\newblock{\em Nuova forma sintetica delle equazioni del moto di un sistema anolonomo ed esistenza di un integrale lineare nelle velocit\`a.} 
\newblock{Boll. Un. Mat. Ital. {\bf 11} (1956), 1--9} (in Italian)

\bibitem{benenti}
S. Benenti, {\em A `user-friendly' approach to the dynamical equations
of non-holonomic systems.} SIGMA Symmetry Integrability Geom. Methods
Appl. {\bf 3} (2007), Paper 036, 33 pp. 

\bibitem{BKMM}
\newblock{A.M. Bloch, P.S. Krishnaprasad, J.E. Marsden and R.M. Murray,}
\newblock{\em Nonholonomic mechanical systems with symmetry.} 
\newblock{Arch. Rational Mech. Anal. {\bf 136} (1996), 21-99. }

\bibitem{BorMam} A.V. Borisov and  I.S. Mamaev,
{\em Rolling of a rigid body on a plane and  sphere. Hierarchy
of dynamics.} Regul. Chaotic Dyn., {\bf 7}  (2002)
177-200.

\bibitem{BorMamK} A.V. Borisov,  I.S. Mamaev and A.A. Kilin,
{\em Rolling of a ball on surface. New integrals ans hierarchy
of dynamics.} Regul. Chaotic Dyn., {\bf 7}  (2002)
201-219. 

\bibitem{BMB2015}
A.V. Borisov, I.S. Mamaev and I.A. Bizyaev,
{\em The Jacobi integral in nonholonomic mechanics. }
Regul. Chaotic Dyn. {\bf 20} (2015), 383-400.

\bibitem{chaplygin}
S.A. Chaplygin, {\em On a motion of a heavy body of revolution on a
horizontal plane}. English translation. Reg. Chaotic Dyn. {\bf7} (2002), 116-130.


\bibitem{CDS}
R. Cushman, J.J. Duistermaat and J. \'Snyaticki,
{\em Geometry of Nonholonomically Constrained Systems.}
 Advanced Series in Nonlinear Dynamics {\bf 26}
(World Scientific, Singapore, 2010).

\bibitem{DallaVia2017}
M. Dalla Via, {\it Geometric and dynamic phase-space structure of a
class of nonholonomic integrable systems with symmetries.} Master
Thesis, University of Padova (2017).

\bibitem{earnshaw}
S. Earnshaw,
{\em Dynamics, or an Elementary Treatise on Motion,} 3d ed.
(Deighton, Cambridge, 1844).

\bibitem{FGS2008}
\newblock{F. Fass\`o, A. Giacobbe and N. Sansonetto,}
\newblock{\em Gauge conservation laws and the momentum
equation in
nonholonomics mechanics.}
\newblock{Rep. Math. Phys. {\bf 62} (2008), 345-367.}

\bibitem{FGS2009}
\newblock{F. Fass\`o, A. Giacobbe and N. Sansonetto,}
\newblock{\em On the number of weakly Noetherian constants
of motion of nonholonomic systems.}
\newblock{J. Geom. Mech. {\bf 1} (2009), 389-416.}


\bibitem{FGS2012}
\newblock{F. Fass\`o, A. Giacobbe and N. Sansonetto,}
\newblock{\em Linear weakly Noetherian constants of motion are horizontal
gauge momenta.}
\newblock{J. Geom. Mech. {\bf 4} (2012), 129-136.}


\bibitem{FRS}
\newblock{F. Fass\`o, A. Ramos and N. Sansonetto,}
\newblock{\em The reaction-annihilator distribution and the
nonholonomic Noether theorem for lifted actions.} 
\newblock{Regul. Chaotic Dyn. {\bf 12} (2007), 449-458.}

\bibitem{FS2009}
\newblock{F. Fass\`o and N. Sansonetto,}
\newblock{\em An elemental overview of the nonholonomic
Noether theorem.}
\newblock{ Int. J. Geom. Methods Mod. Phys. {\bf 6} (2009),
1343-1355.}

\bibitem{FS2015}
F. Fass\`o and N. Sansonetto,
\newblock{\em  Conservation of energy and momenta in nonholonomic
systems with affine constraints. }
Regul. Chaotic Dyn. {\bf 20} (2015), 449-462.

\bibitem{FS2016}
\newblock{F. Fass\`o and N. Sansonetto,}
\newblock{\em Conservation of `moving' energy
in nonholonomic systems with affine constraints
and integrability of spheres
on rotating surfaces.}
J. Nonlinear Sci. 26 (2016), 519-544.


\bibitem{Fedorov} Y. Fedorov {\em On two integrable
nonholonomic problems in classical mechanics.} Vearn.
Moskov. Univ. Ser. I. Mat. Mekh. {\bf 4} (1989), 38-41
(Russian).

\bibitem{fedorov-kozlov} Y.N. Fedorov and V.V. Kozlov, {\em Various
aspects of $n$-dimensional rigid body dynamics}. Amer. Math. Soc.
Transl. Series 2, {\bf 168} (1995), 141-171.

\bibitem{fedorov-jovanovic} Y.N. Fedorov and B. Jovanovi\'c,
{\em Nonholonomic LR systems as generalized Chaplygin systems with
an invariant measure and flows on homogeneous spaces}. J. Nonlinear
Sci. {\bf 14} (2004), 341-381.

\bibitem{ferrario-passerini}
C. Ferrario and A. Passerini, 
\emph{Rolling Rigid Bodies and Forces of Constraint: an Application to
Affine Nonholonomic Systems}. 
Meccanica {\bf 35} (2000), 433-442. 

\bibitem{LGNthesis} L.C. Garc\'ia-Naranjo, {\em Almost
Poisson brackets for nonholonomic systems on Lie groups},
Ph.D. dissertation, University of Arizona (2007).

\bibitem{LGN2007}  L.C. Garc\'ia-Naranjo, {\em Reduction of
almost Poisson brackets for nonholonomic systems on Lie
groups.} Regular Chaotic Dyn. {\bf 14} (2007), 365-388.

\bibitem{LGN2014} L.C. Garc\'ia-Naranjo, A.J. Maciejewski, J.C. Marrero
and M. Przybylska, {\em The inhomogeneous Suslov problem.}
Phys. Lett. A {\bf 378} (2014), 2389--2394.

\bibitem{jotz}
M. Jotz and T. Ratiu,
{\em Dirac structures, nonholonomic systems and reduction}.
Rep. Math. Phys. {\bf 69} (2012), 5-56.

\bibitem{jovanovic2016}
\newblock{B. Jovanovi\'c,}
\newblock{\em Noether symmetries and integrability in Hamiltonian
time-dependent mechanics.}
Theoretical and Applied Mechanics {\bf 43} (2016), 255-273.

\bibitem{jovanovic}
\newblock{B. Jovanovi\'c,}
\newblock{\em Symmetries of line bundles and Noether theorem for
time-dependent nonholonomic systems.} Preprint (2016). 
\newblock{arXiv:1609.01965v1 [math-phys] 7 Sep 2016.}

\bibitem{kobayashi-oliva}
\newblock{M.H. Kobayashi and W.O. Oliva,}
\newblock{\em A note on the conservation of energy and volume in the
setting of nonholonomic mechanical systems.}
\newblock{Qual. Theory Dyn. Syst. {\bf 5}, 4 (2004) 247-259.}

\bibitem{marle2003}
\newblock{C.-M. Marle,}
\newblock{\em On symmetries and constants of motion in Hamiltonian
systems with nonholonomic constraints.} 
\newblock{ In {\em
Classical and Quantum Integrability (Warsaw, 2001)},  223-242,
Banach Center Publ.~{\bf 59} (Polish Acad. Sci. Warsaw, 2003). }

\bibitem{NF}
\newblock{J.I. Neimark and N.A. Fufaev, }
\newblock{\em Dynamics of Nonholonomic Systems.} 
\newblock{Translations of Mathematical Monographs {\bf 33} (AMS,
Providence, 1972).}

\bibitem{Ortega-Ratiu}
J.P. Ortega and T.S Ratiu, {\it Momentum maps and
Hamiltonian
reduction.} Progress in Mathematics {\bf 222} (Birkh\"auser,
Boston, 2004).

\bibitem{pagani91}
E. Massa and E. Pagani,
{\em Classical dynamics of nonholonomic systems: a geometric approach}.
Ann. Inst. H. Poincar\'e Phys. Th\'eor. {\bf 55} (1991), 511-544.

\bibitem{pars}
\newblock{L.A. Pars,}
\newblock{\em A Treatise on Analytical Dynamics.} 
\newblock{(Heinemann, London, 1965).}

\bibitem{routh}
E. J. Routh, {\em Treatise on the Dynamics of a System of Rigid Bodies
(Advanced Part)}. (Dover, New York, 1955).


\bibitem{Veselov-Veselova-1986} A.P. Veselov and L.E. Veselova, 
{\em Flows on Lie groups with nonholonomic constraint and integrable
non-Hamiltonian systems.} Funkt. Anal. Prilozh   {\bf
20} (1986), 65--66 (Russian); English trans.: {\em Funct.
Anal. Appl.} {\bf 20}  (1986),  308--309.

\bibitem{Veselov-Veselova-1988} 
A.P. Veselov and L.E. Veselova, \emph{Integrable
Nonholonomic Systems on Lie Groups}. Mat. Notes {\bf 44} (5-6)
(1988), 810-819.

\bibitem{Veselova} L.E. Veselova, \emph{New cases of integrability
of the equations of motion of a rigid body in the presence of a
nonholonomic constraint.} Geometry, Differential Equations, and
Mechanics (in Russian), Moscow State Univ. (1986), pp. 64-68.

\end{thebibliography}
}
\end{document}